\documentclass[11pt,a4paper]{article}

\usepackage[T1]{fontenc}
\usepackage[margin=1in]{geometry}
\usepackage{authblk}

\usepackage{graphicx}
\usepackage{amsmath,amssymb,amsfonts,amsthm,mathrsfs}
\usepackage{xcolor}
\usepackage{doi}
\usepackage{commath}
\usepackage{algorithm}
\usepackage{algpseudocode}
\usepackage{placeins}
\usepackage[numbers,sort&compress]{natbib}

\usepackage{thmtools}
\usepackage{thm-restate}
\usepackage{hyperref}
\hypersetup{hidelinks}
\usepackage{varioref}
\usepackage{cleveref}

\usepackage{tikz}
\usetikzlibrary{arrows.meta}
\usetikzlibrary{arrows}
\definecolor{vectorblue}{HTML}{1769C2}
\definecolor{trianglegreen}{HTML}{B1DF8F}
\definecolor{triangleedge}{HTML}{93BD7B}

\DeclareMathOperator{\RR}{\mathbb{R}}

\DeclareMathOperator{\QQ}{\mathbb{Q}}

\DeclareMathOperator{\ZZ}{\mathbb{Z}}

\DeclareMathOperator{\BB}{\mathbb{B}}

\DeclareMathOperator{\GC}{\mathcal{G}}
\DeclareMathOperator{\RC}{\mathcal{R}}
\DeclareMathOperator{\BC}{\mathcal{B}}
\DeclareMathOperator{\NC}{\mathcal{N}}
\DeclareMathOperator{\VC}{\mathcal{V}}

\DeclareMathOperator{\TC}{\mathcal{T}}

\DeclareMathOperator{\PC}{\mathcal{P}}

\DeclareMathOperator{\SC}{\mathcal{S}}

\DeclareMathOperator{\HC}{\mathcal{H}}

\DeclareMathOperator{\FC}{\mathcal{F}}
\DeclareMathOperator{\KC}{\mathcal{K}}
\DeclareMathOperator{\UC}{\mathcal{U}}
\DeclareMathOperator{\XC}{\mathcal{X}}
\DeclareMathOperator{\JC}{\mathcal{J}}
\DeclareMathOperator{\IC}{\mathcal{I}}
\DeclareMathOperator{\CCal}{\mathcal{C}}
\DeclareMathOperator{\QC}{\mathcal{Q}}
\DeclareMathOperator{\YC}{\mathcal{Y}}

\DeclareMathOperator{\BS}{\mathscr{B}}

\DeclareMathOperator{\FS}{\mathscr{F}}

\DeclareMathOperator{\NS}{\mathscr{N}}
\DeclareMathOperator{\PS}{\mathscr{P}}

\DeclareMathOperator{\TS}{\mathscr{T}}

\DeclareMathOperator{\rank}{rank}

\DeclareMathOperator{\cone}{cone.\!hull}
\DeclareMathOperator{\conv}{conv\!.\!hull}

\DeclareMathOperator{\linh}{span}

\DeclareMathOperator{\vertex}{vert}

\DeclareMathOperator{\poly}{poly}

\DeclareMathOperator{\relint}{rel.\!int}
\DeclareMathOperator{\inter}{int}

\DeclareMathOperator{\vol}{vol}

\DeclareMathOperator{\diam}{diam}
\DeclareMathOperator{\paral}{Par}

\DeclareMathOperator{\proj}{proj}
\DeclareMathOperator{\extreme}{ext}

\DeclareMathOperator{\dist}{dist}
\DeclareMathOperator{\sphere}{\mathbb{S}}

\DeclareMathOperator{\average}{avg}

\DeclareMathOperator{\BUnit}{\mathbf 1}
\DeclareMathOperator{\BZero}{\mathbf 0}

\newcommand*{\intint}[2][1]{\left\{#1,\, \dots,\, #2\right\}}

\newif\ifhighlightchanges
\highlightchangestrue
\newcommand{\gribHighlightColor}[1]{\ifhighlightchanges\color{#1}\fi}

\newenvironment{ChangedText}
{\par\gribHighlightColor{blue}}
{}

\newcommand\restr[2]{{
  \left.\kern-\nulldelimiterspace 
  #1 
  \vphantom{\big|} 
  \right|_{#2} 
  }}
\highlightchangesfalse

\DeclareMathOperator{\subdiv}{Subdiv}
\DeclareMathOperator{\facefan}{face.\!fan}
\DeclareMathOperator{\normfan}{norm.\!fan}

\theoremstyle{plain}
\newtheorem{theorem}{Theorem}
\newtheorem{lemma}[theorem]{Lemma}
\newtheorem{proposition}[theorem]{Proposition}
\newtheorem{corollary}[theorem]{Corollary}

\theoremstyle{remark}
\newtheorem{example}{Example}
\newtheorem{remark}{Remark}

\theoremstyle{definition}
\newtheorem{definition}{Definition}

\allowdisplaybreaks

\begin{document}

\title{On the Vertices of Delta-modular Polyhedra}

\author[1]{Mikhail~Bludov\thanks{\texttt{michaelbludov@gmail.com}}}
\author[2,1]{Dmitry~Gribanov\thanks{Corresponding author: \texttt{dimitry.gribanov@gmail.com}}}
\author[1]{Maxim~Klimenko\thanks{\texttt{klimkomx@gmail.com}}}
\author[1]{Andrey~Kupavskii\thanks{\texttt{kupavskii@ya.ru}}}
\author[3,4]{Zsolt~L\'angi\thanks{\texttt{zlangi@server.math.u-szeged.hu}}}
\author[1,5]{Vsevolod~Voronov\thanks{\texttt{v-vor@yandex.ru}}}
\author[1]{Alexander~Rogozin\thanks{\texttt{aleksandr.rogozin@phystech.edu}}}

\affil[1]{Laboratory of Discrete and Combinatorial Optimization, Moscow Institute of Physics and Technology, Institutsky lane 9, Dolgoprudny, 141700, Moscow region, Russian Federation}
\affil[2]{Laboratory of Algorithms and Technologies for Network Analysis, HSE University, Rodionova Ulitsa 136, Nizhny Novgorod, 603093, Nizhny Novgorod region, Russian Federation}
\affil[3]{Bolyai Institute, University of Szeged, 13 Dugonics square, Szeged, H-6720, Hungary}
\affil[4]{HUN-REN Alfr\'ed R\'enyi Institute of Mathematics, Re\'altanoda utca 13-15, Budapest, 1053, Hungary}
\affil[5]{Caucasus Mathematical Center, Adyghe State University, Pervomayskaya, 208, Maykop, 385000, Republic of Adygea, Russian Federation}
\date{}

\maketitle

\begin{abstract}
Let $\PC$ be a polytope defined by the system $A x \leq b$, where $A \in \RR^{m \times n}$, $b \in \RR^m$, and $\rank A = n$. We give a short geometric proof of the following tight upper bound on the number of vertices of $\PC$:
\begin{equation*}
    \abs{\vertex(\PC)} \leq n! \cdot \frac{\Delta}{\Delta_{\average}} \cdot \vol \BB_2 \sim \frac{1}{\sqrt{\pi n}} \cdot \left(\frac{2 \pi}{e}\right)^{n/2} \cdot n^{n/2} \cdot \frac{\Delta}{\Delta_{\average}},
\end{equation*}
where $\Delta$ is the maximum absolute value of $n \times n$ subdeterminants of $A$, and $\Delta_{\average}$ is the average absolute value of subdeterminants of $A$ corresponding to a triangulation of $\PC$'s normal fan. Assuming that $A$ is integer, such polyhedra are called $\Delta$-modular polyhedra. Note that in the integer case, the bound can be simplified via the inequality $\Delta_{\average} \geq \Delta_{\min} \geq 1$, where $\Delta_{\min}$ is the minimum absolute value of subdeterminants of $A$ corresponding to feasible bases of $A x \leq b$. For this, we prove and use a variant of Macbeath's theorem on approximation of centrally symmetric convex bodies by centrally symmetric polytopes.

Additionally, we give a direct argument based on prior results in the field, showing that the graph diameter of $\PC$ is bounded by 
$O\bigl(n^3 \cdot \frac{\Delta}{\Delta_{\min}} \cdot \ln (n \frac{\Delta}{\Delta_{\min}}) \bigr)$. 
Thus, both characteristic of $\PC$ are linear in $\Delta/\Delta_{\min}$. 
Previously, only quadratic bounds were known in this parameter.

From an algorithmic perspective, we demonstrate that:
\begin{enumerate}
    \item Given $A \in \QQ^{m \times n}$, $b \in \QQ^m$, and an initial feasible solution to $A x \leq b$, the convex hull of $\PC$ can be constructed in $O(n)^{n/2} \cdot m^2 \cdot \frac{\Delta}{\Delta_{\average}}$ operations. For simple polyhedra, the dependence on $m$ reduces to linear;
    
    \item Given $A \in \ZZ^{m \times n}$ and $b \in \QQ^m$, the integer point count $\abs{\PC \cap \ZZ^n}$ can be computed in $O(n)^n \cdot \frac{\Delta^4}{\Delta_{\average}}$ arithmetic operations.
\end{enumerate}
\end{abstract}


\clearpage
\tableofcontents

\section{Introduction}

Let $\PC$ be a polyhedron defined by the system $A x \leq b$, where $A \in \RR^{m \times n}$ and $b \in \RR^m$. Throughout our work, we consider only \emph{full-dimensional} and \emph{pointed polyhedra}, meaning $\dim(\PC) = n$ and $\PC$ contains no lines. The latter condition is equivalent to $\rank(A) = n$.
Additionally, we assume all inequalities in the system $A x \leq b$ are facet-defining, meaning the system contains no redundant inequalities. In these assumptions, $(n,m)$ denotes the dimension and number of facets in $\PC$ respectively.

The seminal Upper Bound Theorem by McMullen \cite{MaxFacesTh} states that the vertex count reaches its maximum for the class of polytopes dual to cyclic polytopes. Combining this with the formula from \cite[Section~4.7]{Grunbaum} for the facet count of cyclic polytopes yields:
\begin{equation}\label{eq:McMullenUB}
    \abs{\vertex(\PC)} = O\left(\frac{m}{n}\right)^{n/2}.
\end{equation}

But, in many scenarios and especially in the context of ILP, the resulting polyhedra are very far from cyclic. In many cases, it is useful to look at the parameterization by $\Delta/\Delta_{\average}$ or $\Delta/\Delta_{\min}$, where the parameters $\Delta$, $\Delta_{\min}$ and $\Delta_{\average}$ are defined by the following way:
\begin{gather}
    \Delta(A) = \max\left\{ \abs{\det A_{\BC}} \colon \BC \in \binom{\intint m}{n} \right\}, \notag\\
    \Delta_{\average}(\PC,A) = \max \left\{ \frac{1}{\abs{\TS}} \sum_{\BC \in \TS} \abs{\det A_{\BC}} \colon \text{$\TS$ is a triangulation of $\PC$'s normal fan} \right\}, \label{eq:delta_avg_starting}\\
    \Delta_{\min}(\PC,A) = \max\left\{ \min\limits_{\BC \in \TS} \abs{\det A_{\BC}} \colon \text{$\TS$ is a triangulation of $\PC$'s normal fan} \right\}. \label{eq:delta_min_starting}
\end{gather}
Here we identify the simplicial cones which form a triangulation $\TS$ with the corresponding feasible bases of the system $A x \leq b$ (see Section \nameref{sec:prelim} for the formal definitions of the normal fan of $\PC$ and its triangulation).

\begin{remark}\label{rm:delta_avg_def}
    Assuming that $\PC$ is simple, the normal fan of $\PC$ is simplicial and coincides with the set of feasible bases of the system $A x \leq b$ denoted by $\BS$. In this case, the definitions of $\Delta_{\min}$ and $\Delta_{\average}$ can be simplified by the following way:
    \begin{gather*}
        \Delta_{\average}(\PC,A) = \frac{1}{\abs{\BS}} \sum_{\BC \in \BS} \abs{\det A_{\BC}},\\
        \Delta_{\min}(\PC,A) = \min\limits_{\BC \in \BS} \abs{\det A_{\BC}}.
    \end{gather*}
\end{remark}

As our first contribution, we establish the following bound for the vertex count of $\PC$ (the proof is given in \Vref{proof:th:vertexUB}).
\begin{restatable}{theorem}{vertexUB}
\label{th:vertexUB}
    Let $\PC \subseteq \RR^n$ be a pointed $n$-dimensional polyhedron defined by a system $A x \leq b$. Denote $\Delta = \Delta(A)$ and $\Delta_{\average} = \Delta_{\average}(\PC,A)$, then  
    \begin{multline*}
        \abs{\vertex(\PC)} \leq n! \cdot \frac{\Delta}{\Delta_{\average}} \cdot \vol(\BB_2^n) \sim \\
        \sim \frac{1}{\sqrt{\pi n}} \cdot \left(\frac{2 \pi}{e}\right)^{n/2} \cdot n^{n/2} \cdot \frac{\Delta}{\Delta_{\average}} = O(n)^{n/2} \cdot \frac{\Delta}{\Delta_{\average}}.
    \end{multline*}
\end{restatable}
Here $\BB_2^n$ denotes the standard Euclidean ball.
This inequality is tight in the following sense (the proof is given in \Vref{proof:prop:vertexLB}). 
\begin{restatable}{proposition}{vertexLB}
\label{prop:vertexLB}
    There exists a sequence of pointed $n$-dimensional polytopes $\PC_k \subseteq \RR^n$ defined by systems $A_k x \leq b_k$ such that:
    \begin{gather*}
        \vertex(\PC_k) \underset{k \to \infty}{\sim} n! \cdot \frac{\Delta_k}{\Delta_{\average,k}} \cdot \vol \BB_2^n, \quad\text{and}\\
        \lim\limits_{k \to \infty} \frac{\Delta_k}{\Delta_{\average,k}} = \infty,
    \end{gather*}
    where $\Delta_k = \Delta(A_k)$ and $\Delta_{\average,k} = \Delta_{\average}(\PC_k,A_k)$. 
\end{restatable}
In fact, motivated by ILP applications described in \Cref{sec:algo_implications}, we establish a more general result concerning the count of $n$-dimensional faces in polyhedral fans (see \Cref{cor:vol_Delta_fan} and \Cref{prop:tightness_fans}). The result follows by utilizing properties of `totally unimodular' convex bodies described in \Cref{prop:totally_unimodular_set}, followed by an application of a symmetric version of Macbeath's theorem \cite{Macbeath}. In original form, Macbeath's theorem states that the Euclidean ball is the hardest to approximate by polytopes. The symmetric version formulated and proved as \Cref{th:Macbeath} states that the Euclidean ball is the hardest to approximate even by centrally symmetric polytopes.

We can compare \Cref{th:vertexUB} with the McMullen upper bound \eqref{eq:McMullenUB} by the following way. Assuming that $A$ is an integer matrix, it is known that $m$ can be bounded by a function depending on $n$ and $\Delta$. For instance:
\begin{itemize}
    \item $m = O(n^2 \cdot \Delta^2)$ (Lee~et~al.~\cite{ModularDiffColumns})
    \item $m = O(n^4 \cdot \Delta)$ (Averkov \& Schymura~\cite{DiffColumnsOther})
    \item $m = O(n^2 + n \cdot \Delta^7)$ (Paat~et~al.~\cite{DifferingSeparated})
\end{itemize}
Furthermore, \cite{ModularDiffColumns} provides examples of polyhedra with $m = \Theta(n^2 + n \cdot \Delta)$. Substitution of one of these expressions into McMullen's upper bound \eqref{eq:McMullenUB} provides upper bounds for $\abs{\vertex(\PC)}$ depending only on $n$ and $\Delta$. However, even in the case $m = \Theta(n^2 + n \cdot \Delta)$, \Cref{th:vertexUB} works better.

Another problem related to the use of parameter $\Delta$ is the problem of estimating the diameter of a polyhedral graph $\GC(\PC)$. Recall that graph $\GC(\PC)$ is the $1$-skeleton of polyhedron $P$, i.e., the graph formed by its vertices and edges. With respect to our result about the number of vertices, the results concerning the diameter of $\GC(\PC)$ use the properties of "totally unimodular" convex bodies described in \Cref{prop:totally_unimodular_set} in conjunction with the results from \cite{SubdeterminantsDiameter,ShadowSimplexForCurved}. As the main result in this direction, we present the following theorem (the proof is given in \Vref{sec:diam_UB}).
\begin{restatable}{theorem}{diamUB}
\label{th:diamUB}
    Let $\PC \subseteq \RR^n$ be a pointed $n$-dimensional polyhedron defined by a system $A x \leq b$. Denote $\Delta = \Delta(A)$ and $\Delta_{\min} = \Delta_{\min}(\PC,A)$, then
    \begin{equation*}
        \diam\bigl(\GC(\PC)\bigr) = O\left(n^3 \cdot \frac{\Delta}{\Delta_{\min}} \cdot \ln \bigl(n \frac{\Delta}{\Delta_{\min}}\bigr) \right).
    \end{equation*}
\end{restatable}

Unfortunately, compared to the upper bound $O\left(n^3 \cdot \Delta^2 \cdot \ln (n \Delta) \right)$ for integer matrices\footnote{In precise form, the upper bound by \cite{ShadowSimplexForCurved} is $O\left(n^3 \cdot \Delta_{n-1} \cdot \Delta_1 \cdot \ln (n \Delta_{n-1} \Delta_1) \right)$, where $\Delta_k$ (for $k \in \intint n$) denotes the maximum absolute value of $k \times k$ subdeterminants of $A$. In our work we replace $\Delta_{n-1} \cdot \Delta_1$ by $\Delta_n$.} by Dadush \& H\"ahnle \cite{ShadowSimplexForCurved}, our upper bound is non-constructive\footnote{When discussing diameter computation, we additionally assume $A$ and $b$ are rational.}, as it requires knowledge of a submatrix $B \in \RR^{n \times n}$ of $A$ with $\abs{\det B}$ close to $\Delta$ - a computationally difficult task (see, e.g., Nikolov \cite{LargestSimplex_Nikolov}). We give a small survey about the diameter problem in \Cref{sec:related_work}.

\begin{ChangedText}
    The dependence on $\frac{\Delta}{\Delta_{\min}}$ is tight, as will be shown in the next proposition. The proof could be found in \Vref{proof:prop:diamLB_DeltaMin}.

    \begin{restatable}{proposition}{diamLBDeltaMin}\label{prop:diamLB_DeltaMin}
    For every fixed $n \geq 2$, there exists a sequence of simple $n$-dimensional polytopes $\PC_k\subseteq\RR^n$, $k\geq 1$, defined by systems $A_k x\leq b_k$, such that
    \begin{gather*}
        \Delta_k = 1, \qquad
        \Delta_{\min,k} = \Delta_{\average,k} = \sin\frac{\pi}{2k},\\
        \diam\bigl(\GC(\PC_k)\bigr) = 2k+n-2
        \;\underset{k \to \infty}{\sim}\;
        \pi \cdot \frac{\Delta_k}{\Delta_{\min,k}},
        \qquad
        \lim_{k \to \infty} \frac{\Delta_k}{\Delta_{\min,k}} = \infty,
    \end{gather*}
    where $\Delta_k=\Delta(A_k)$, $\Delta_{\min,k}=\Delta_{\min}(\PC_k,A_k)$, and $\Delta_{\average,k}=\Delta_{\average}(\PC_k,A_k)$. The same asymptotic equality holds with $\Delta_{\average,k}$ in place of $\Delta_{\min,k}$.
\end{restatable}

Since the determinant ratios tend to infinity, \Cref{prop:diamLB_DeltaMin} rules out any general bound of the form $\diam\bigl(\GC(\PC)\bigr)\leq C_n(\Delta/\Delta_{\min})^\alpha$ with $\alpha<1$, even when $C_n$ may depend arbitrarily on $n$. Thus the exponent of $\Delta/\Delta_{\min}$ in \Cref{th:diamUB} is optimal; the logarithmic factor and the dependence on $n$ are not settled by this example. The same obstruction holds for $\Delta/\Delta_{\average}$ and, by \Cref{prop:fan_diam_LB_linear}, for graphs of polyhedral fans.
\end{ChangedText}

\begin{ChangedText}
It is surprising that, for the parameter $\frac{\Delta}{\Delta_{\average}}$, one can construct a lower bound that is exponential in $n$ but only sublinear in $\frac{\Delta}{\Delta_{\average}}$. Unfortunately, we are not aware of any diameter upper bounds depending on it. This remains an interesting open question. The proof is given in \Vref{proof:prop:diamLB_DeltaAvg}. The case of polehedral fans is considered in \Cref{prop:fan_diamLB_DeltaAvg}.
\begin{restatable}{proposition}{diamLBDeltaAvg}
\label{prop:diamLB_DeltaAvg}
    Suppose that $n \geq 3$. There exists a sequence of pointed $n$-dimensional polytopes $\PC_k \subseteq \RR^n$ defined by systems $A_k x \leq b_k$ such that:
    \begin{gather*}
        \diam\bigl(\GC(\PC_k)\bigr) + 1 = 2^{k+1}\;\underset{k \to \infty}{\sim}\; C_n \cdot \left( \frac{\Delta_k}{\Delta_{\average,k}} \right)^{\frac{1}{\log_2 n}},\\
        \text{where}\quad C_n = \left(n! \cdot \vol \BB_2^n \right)^{\frac{1}{\log_2 n}} \;\underset{n \to \infty}{\sim}\; \left( \frac{2 \pi n}{e} \right)^{\frac{n}{2 \log_2 n}}.
    \end{gather*}
    Here, $\Delta_k = \Delta(A_k)$ and $\Delta_{\average,k} = \Delta_{\average}(\PC_k,A_k)$.
\end{restatable}
\end{ChangedText}

\subsection{Algorithmic Implications}\label{sec:algo_implications}

One of our motivations comes from the problem of efficient counting of integer points inside polyhedra. Many known algorithms for this problem are based on the construction of a normal fan triangulation corresponding to $\PC$. In our work we provide a variant of such an algorithm, whose computational complexity is linear with respect to the size of an output triangulation. In this way, our upper bound on size of polyhedral fans (see \Cref{cor:vol_Delta_fan}) yields the following result that can be interesting in its own. 
\begin{restatable}{theorem}{vertexEnumeration}
\label{th:vertex_enumeration}
    Let $\PC$ be a pointed polyhedron defined by the system $A x \leq b$, where $A \in \QQ^{m \times n}$ and $b \in \QQ^m$. Then, there exists an algorithm that enumerates the vertices of $\PC$ in
    \begin{equation*}
        O(n)^{n/2} \cdot m^2 \cdot \frac{\Delta}{\Delta_{\average}} + T_{\text{LP}} \quad\text{operations},
    \end{equation*}
    where $T_{\text{LP}}$ denotes the number of operations required to compute a rational feasible solution of $A x \leq b$. If $\PC$ is simple, the complexity bound improves to
    \begin{equation*}
        O(n)^{n/2} \cdot m \cdot \frac{\Delta}{\Delta_{\average}} + T_{\text{LP}} \quad\text{operations}.
    \end{equation*}
\end{restatable}
Assuming that $A$ is integer, we can replace $\frac{\Delta}{\Delta_{\average}}$ with $\Delta$, which consequently allows to eliminate the $T_{\text{LP}}$ term.

It follows from the complexity bound that
\begin{equation*}
    T_{\text{LP}} = n^{O(1)} \cdot m \cdot \log\left(1 + \Delta\right),
\end{equation*}
established by Dadush, Natura \& V\'egh \cite{TardosRevisiting_Dadush}. 

Our main algorithmic result, which also serves as one of the primary motivations for this paper, is the following.
\begin{restatable}{theorem}{integerCounting}
\label{th:integer_counting}
    Let $\PC$ be a pointed polyhedron defined by the system $A x \leq b$, where $A \in \ZZ^{m \times n}$ and $b \in \QQ^m$. Then, there exists an algorithm that computes $\abs{\PC \cap \ZZ^n}$ in
    \begin{equation*}
        O(n)^n \cdot \frac{\Delta^4}{\Delta_{\average}} \quad \text{operations.}
    \end{equation*}
\end{restatable}
When compared to Barvinok's algorithm \cite{BarvPom,BarvBook}, whose complexity bound can be estimated in
\begin{equation*}
    O(m/n)^{n/2} \cdot \bigl(\ln \Delta\bigr)^{n \ln n} \quad \text{operations},
\end{equation*}
our bound demonstrates superiority in scenarios where:
\begin{itemize}
    \item $\Delta = n^{O(n \ln^{1 - \varepsilon} n)}$ for $\varepsilon > 0$ and $m \geq n$, or
    \item $m = \Omega(n \Delta)$ with $\Delta \geq n^{2 + \varepsilon}$.
\end{itemize}

\subsection{Comparison with Other Parameters}\label{sec:intro:comparison}

\begin{ChangedText}
    Besides $\Delta/\Delta_{\average}$ and $\Delta/\Delta_{\min}$, there exist other important parameters that significantly influence the properties and geometry of polyhedra. These parameters are captured by the notions of the \emph{local $\delta$-distance property} and \emph{$\tau$-wide polyhedra}. The local $\delta$-distance property was introduced by Bonifas et al. \cite{SubdeterminantsDiameter}, while the concept of $\tau$-wide polyhedra was introduced by Dadush and H\"ahnly \cite{ShadowSimplexForCurved}.
    The corresponding definitions are provided below. 
    
    \begin{definition}[$\delta$-distance Property]\label{def:polyhedral_delta_distance}
        A set of linearly independent vectors $v_1, \dots, v_k \in \RR^n$ satisfies the \emph{$\delta$-distance property} if for every $i \in \intint k$, the vector $v_i$ is at Euclidean distance at least $\delta \norm{v_i}_2$ from $\linh\left(v_j \colon j \in \intint k \setminus i\right)$.

        For a polyhedron $\PC = \bigl\{ x \in \RR^n \colon A x \leq b \bigr\}$ with $\rank(A) = n$, we define $\PC$ to satisfy the \emph{local $\delta$-distance} property if every feasible basis of $A$, i.e. the rows of $A$ defining a vertex of $\PC$, satisfies the $\delta$-distance property.
    \end{definition}

    \begin{definition}[$\tau$-wide Polyhedra]\label{def:polyhedral_tau_wide}
        We say that a cone $\CCal$ is \emph{$\tau$-wide} if it contains a Euclidean ball of radius $\tau$ centered on the unit sphere. We define a polyhedron $\PC$ to have a \emph{$\tau$-wide} normal fan (or simply $\PC$ to be $\tau$-wide) if every vertex normal cone is $\tau$-wide. (For formal definitions of normal cones and normal fans, see \Cref{sec:prelim}.)
    \end{definition}

    The local $\delta$-distance property implies $\tau$-width with $\tau\geq\delta/n$ (see \cite[Lemma~5]{ShadowSimplexForCurved} \Cref{lm:diam_prop_transition}), but the converse inequality $\delta\geq\tau$ does not hold in general. The following proposition gives a comparison under an additional geometric assumption. To avoid interrupting the flow of the exposition, we defer the proof to \Vref{proof:prop:tau_delta_comparison}.

    \begin{restatable}{proposition}{tauDeltaComparison}\label{prop:tau_delta_comparison}
        Let $r_1,\dots,r_n\in\sphere^{n-1}$ be linearly independent, and set $\CCal=\cone(r_1,\dots,r_n)$. Denote by $\delta$ the largest distance parameter of this basis and by $\tau$ the largest width parameter of $\CCal$. Suppose there exist a unit vector $v$ and $\alpha>0$ such that
        \[
            \langle v,r_i\rangle\geq\alpha\qquad\text{for all }i\in\intint n.
        \]
        Then
        \[
            \boxed{\frac{\delta}{n}\leq\tau\leq\frac{\delta}{\alpha}.}
        \]
    \end{restatable}

    In the polytope families considred in our paper, every normal cone is simplicial and lies either in a coordinate orthant or in a cone over a facet of $[-1,1]^n$ or a centered regular simplex. In an orthant, take $v$ to be its unit diagonal; for a cone over a facet, take the outward unit facet normal. These choices provide $\alpha\geq1/n$. Subdividing these cones preserves the same bound. 

    Consequently, taking the minimum over all vertex normal cones gives, for the largest admissible parameters of these polytopes,
    \[
        \boxed{\frac{\delta}{n}\leq\tau\leq n\delta.}
    \]
    Thus, for these families, $\tau$ and $\delta$ differ by at most a factor depending only on $n$. By this reason, in the remainder of this subsection, we focus on comparing $\delta$ with $\Delta/\Delta_{\average}$ and $\Delta/\Delta_{\min}$.


    A strong point of the parameter $\delta$ is that, in contrast to $\Delta/\Delta_{\average}$ and $\Delta/\Delta_{\min}$, it is independent of the row scalings of the system $A x \leq b$. Moreover, one can construct upper bounds on the number of vertices and the diameter of $\PC$ that depend only on $n$ and $\delta$, which is stated in the following \Cref{prop:vertexUB_delta_local}. The proof follows in a straightforward manner from the fact that the normal cones of the polyhedron have pairwise disjoint interiors, and that the volume of the intersection of any such simplicial cone with the Euclidean sphere is bounded below by the volume of the simplex spanned by the normal vectors generating the rays of the cone. The colume of each such simple is lower bounded $\delta^{n-1}/n!$. We skip the formal proof here, as we will later derive it in this section as a corollary of our main \Cref{th:vertexUB}.
    \begin{proposition}\label{prop:vertexUB_delta_local}
        Let $\PC \subseteq \RR^n$ be a pointed $n$-dimensional polyhedron defined by a system $A x \leq b$ which satisfyes the local $\delta$-distance property. Then,
        \begin{equation*}
            \abs{\vertex(\PC)} \leq n! \cdot \vol(\BB_2^n) \cdot \frac{1}{\delta^{n-1}}.
        \end{equation*}
    \end{proposition}

    As for the diameter, the bound is given in the following theorem, which was proved by Dadush and H\"ahnly \cite{ShadowSimplexForCurved}.
    \begin{theorem}[Dadush and H\"ahnly \cite{ShadowSimplexForCurved}]\label{th:diamUB_delta_local}
        Let $\PC \subseteq \RR^n$ be a pointed $n$-dimensional polyhedron defined by a system $A x \leq b$ which satisfyes the local $\delta$-distance property. Then,
        \begin{equation*}
            \diam\bigl(\GC(\PC)\bigr) = O\left(\frac{n^2}{\delta} \cdot \ln \left( 1 + \frac{n}{\delta} \right)\right).
        \end{equation*}
    \end{theorem}
    
    Conversely, a strong point of $\Delta/\Delta_{\average}$ and $\Delta/\Delta_{\min}$ is their invariance under non-degenerate affine transformations of the coordinate space. Consequently, the connection between these parameters and a local $\delta$-distance parameter is far from trivial. We will attempt to clarify this relationship below.

    We first provide an example that clearly demonstrates that the parameters $\Delta/\Delta_{\average}$ and $\Delta/\Delta_{\min}$ do depend on scaling. Consequently, a poorly chosen scaling may substantially degrade the bounds obtained with these parameters in comparison to the bounds obtained from $\delta$. Nevertheless, we will show below that a scaling can always be chosen so as to yield a bound at least as strong as that of \Cref{prop:vertexUB_delta_local}.

    \begin{example}\label{ex:delta_row_scaling}
        Fix $M \geq 1$, and consider the system
        \begin{equation}\label{eq:M_example}
            \begin{cases}
                M x_1 \leq M\\
                M x_2 \leq M\\
                x_1+x_2 \leq 1 \\
                -\BUnit \leq x \leq \BUnit\\
                x \in \RR^n.
            \end{cases}
        \end{equation}
        Let $\PC$ be the polytope defined by this system, and $\QC$ be a $5$-gon defined by the convex hull of the points $\left\{(-1,-1), (-1,1), (1,-1), (1,0), (0,1)\right\}$. It is straightforward to see that $\PC = \QC \otimes [-1,1]^{n-2}$, and it does not depend on $M$. 

        
        The five feasible pairs of planar normals (corresponding to the vertices of $\QC$) have determinant absolute values
        \[
            M,\ M,\ M,\ 1,\ M.
        \]
        The remaining coordinate directions contribute determinant $1$.
        The pair $(M e_1,M e_2)$ is infeasible and attains the largest determinant. Hence
        \begin{gather*}
            \Delta(A)=M^2, \qquad
            \Delta_{\min}(\PC,A)=1, \qquad
            \Delta_{\average}(\PC,A)=\frac{4M+1}{5},\\
            \frac{\Delta(A)}{\Delta_{\min}(\PC,A)}=M^2, \qquad
            \frac{\Delta(A)}{\Delta_{\average}(\PC,A)}
            =\frac{5M^2}{4M+1}\underset{M \to \infty}{\sim}\frac54 M.
        \end{gather*}
        The smallest angle between adjacent planar normals in $\QC$ is $\pi/4$, and the other normal directions of $\PC$ are orthogonal to the normal directions of $\QC$. Thus, the local distance parameter is $\delta=1/\sqrt{2}$, independently of $M$. Moreover,
        \[
            \abs{\vertex(\PC_M)}=5\cdot 2^{n-2}, \qquad
            \diam\bigl(\GC(\PC_M)\bigr)=2+(n-2)=n.
        \]
        For fixed $n$, the vertex bound from \Cref{th:vertexUB} grows linearly with $M$, and the diameter bound from \Cref{th:diamUB} has the form $O(M^2\ln M)$. In contrast, bounds depending only on $n$ and $\delta$ from \Cref{prop:vertexUB_delta_local} and \Cref{th:diamUB_delta_local} stay bounded. 
        
        Dividing the first and second inequalities by $M$ gives a representation in which $\Delta=\Delta_{\min}=\Delta_{\average}=1$. 
        Now, one can see that our bound on the number of vertices in terms of $\Delta/\Delta_{\average}$ (\Cref{th:vertexUB}) improves upon the upper bound in terms of $\delta$ (\Cref{prop:vertexUB_delta_local}):
        \begin{equation*}
            \abs{\vertex(\PC)} \;\underset{\Cref{th:vertexUB}}{\leq}\; n! \cdot \vol(\BB_2^n); \qquad \abs{\vertex(\PC)} \;\underset{\Cref{prop:vertexUB_delta_local}}{\leq}\; n! \cdot \vol(\BB_2^n) \cdot 2^{(n-1)/2}.
        \end{equation*}
        But the both bounds are far from tight.

        On the other hand, regarding the diameter bounds with respect to the two parameters $\Delta/\Delta_{\min}$ and $\delta$, the bound in our \Cref{th:diamUB} is inferior to the one in \Cref{th:diamUB_delta_local}:
        \begin{equation*}
            \diam\bigl(\GC(\PC)\bigr) \;\underset{\Cref{th:diamUB}}{=}\; O(n^3 \cdot \log n); \qquad
            \diam\bigl(\GC(\PC)\bigr) \;\underset{\Cref{th:diamUB_delta_local}}{=}\; O(n^2 \cdot \log n).
        \end{equation*}
        Again, these bounds are not tight in the current case.

    \end{example}

    \subsubsection{The Strangthened Version of Determinant-Based Parameters}

    \Cref{ex:delta_row_scaling} demonstrated that, for an appropriate choice of row scaling of the system $A x \leq b$, the bound of \Cref{th:vertexUB} based on $\Delta/\Delta_{\average}$ may be stronger than the bound based on the local $\delta$-distance property (\Cref{prop:vertexUB_delta_local}). Nevertheless, it is easy to show that this phenomenon is not specific to this example and in fact always occurs: it suffices to normalize the rows of the matrix $A$. Indeed, in this case we clearly have
    \begin{gather}
        \Delta \leq 1, \qquad \Delta_{\average} \geq \Delta_{\min} \geq \frac{1}{\delta^{n-1}}, \qquad\text{and} \notag\\   
        \abs{\vertex(\PC)} \;\underset{\Cref{th:vertexUB}}{\leq}\; n! \cdot \vol(\BB_2^n) \cdot \frac{\Delta}{\Delta_{\average}} \quad\leq\quad n! \cdot \vol(\BB_2^n) \cdot \frac{1}{\delta^{n-1}}. \label{eq:vertexUB_delta_local_as_corollary}
    \end{gather}
    Consequently, \Cref{prop:vertexUB_delta_local} is in fact a corollary of our main \Cref{th:vertexUB}.

    This fact motivates us to design a strengthened version of the parameter $\Delta/\Delta_{\average}$ (as well as $\Delta/\Delta_{\min}$) which, unlike the original, depends only on the polyhedron $\PC$ and not on the row scaling of $A x \leq b$.

    \begin{definition}
        Let $\PC \subseteq \RR^n$ be a pointed $n$-dimensional polyhedron. We define the \emph{strengthened determinant ratio} $\kappa_{\PC}$ of $\PC$ as follows.
        \begin{itemize}
            \item Let $A x \leq b$ be an arbitrary system defining $\PC$ in which every inequality defines a facet (in other words, there are no redundant inequalities);
            
            \item Set
            \begin{equation*}
                \kappa_{\PC} = \inf_{D} \frac{\Delta(D A)}{\Delta_{\average}(\PC, D A)},
            \end{equation*}
            where the infimum is taken over all diagonal matrices $D \in \RR^{m \times m}$ with positive diagonal entries;
            
            \item It is easy to see that $\kappa_{\PC}$ is well defined and does not depend on the choice of the system $A x \leq b$, because all systems $D A x \leq D b$ define the same polyhedron.
        \end{itemize}
    \end{definition}

    We can now state strengthened versions of our main Theorems~\ref{th:vertexUB} and~\ref{th:diamUB}.
    \begin{theorem}\label{th:vertexUB_strengthened}
        Let $\PC \subseteq \RR^n$ be a pointed $n$-dimensional polyhedron. Then
        \begin{gather*}
            \abs{\vertex(\PC)} \leq n! \cdot \vol(\BB_2^n) \cdot \kappa_{\PC},\\
            \diam\bigl(\GC(\PC)\bigr) = O\!\left(n^3 \cdot \kappa_{\PC} \cdot \log \bigl(n \kappa_{\PC}\bigr) \right).
        \end{gather*}
    \end{theorem}

    In particular, by \eqref{eq:vertexUB_delta_local_as_corollary}, this also strengthens the local $\delta$-distance bound of \Cref{prop:vertexUB_delta_local} for $\abs{\vertex(\PC)}$.
    \begin{proposition}\label{prop:delta_kappa}
        Let $\PC \subseteq \RR^n$ be a pointed $n$-dimensional polyhedron defined by a system $A x \leq b$ which satisfyes the local $\delta$-distance property. Then,
        \begin{equation*}
            \kappa_{\PC} \leq \frac{1}{\delta^{n-1}}.
        \end{equation*}
    \end{proposition}

\end{ChangedText}

\subsubsection{Additionall Examples and Final Comparison}

    \begin{ChangedText}
        In this subsection, we present several further examples. As for the bound on the number of vertices, the sequences of polytopes in \Cref{prop:vertexLB} and \Cref{prop:diamLB_DeltaMin} show that the main inequality in \Cref{th:vertexUB} can be exponentially stronger than the bound in \Cref{prop:vertexUB_delta_local}. However, the last example given in \Cref{prop:spherical_delta_comparison} below shows that the bound in \Cref{prop:vertexUB_delta_local} can also be optimal in its dependence on the $\delta$-distance parameter. As for the diameter bound, despite the sharpness of our estimate in \Cref{th:diamUB}, the bound of \Cref{th:diamUB_delta_local} due to Dadush and H\"ahnle \cite{ShadowSimplexForCurved} is always at least as good.

        Throughout this subsection, the comparisons concern the representations of the families considered here. We write $O_n$ and $\Theta_n$ when the implicit constants may depend only on $n$. Additionally, in this subsection, we denote $\gamma_n=n!\cdot\vol(\BB_2^n)$.
    \end{ChangedText}

    \begin{ChangedText}
        First consider the family from \Cref{prop:vertexLB}. By \Cref{rm:vertexLB_local_delta}, its largest local distance parameter is $\delta_k=\Theta_n(n^{-k})$.
        By construction, $\Delta_{\min,k}$ and $\Delta_{\average,k}$ are both $\Theta_n(n^{-k})$. 
        Since $\Delta_k\to1$, we obtain
        \[
            \delta_k=\Theta_n(n^{-k}),\qquad \frac{\Delta_k}{\Delta_{\min,k}}=\Theta_n(n^k),\qquad
            \frac{\Delta_k}{\Delta_{\average,k}}\sim \frac{n+1}{\gamma_n}n^k.
        \]
        The number of vertices is exactly $(n+1)n^k$, and the bounds of \Cref{th:vertexUB} are asymptotically equal to this number. In contrast, \Cref{prop:vertexUB_delta_local} gives a bound of order $n^{k(n-1)}$. For fixed $n\geq3$, the latter bound is larger by a factor $\Theta_n(n^{k(n-2)})$, which grows exponentially with $k$. For the diameter, both \Cref{th:diamUB} and \Cref{th:diamUB_delta_local} give $O_n(k n^k)$, whereas the exact value is $2^{k+1}-1$ for $n\geq3$.

        Next, the polytopes from \Cref{prop:diamLB_DeltaMin} are products of a regular $4k$-gon and an $(n-2)$-dimensional cube. By \Cref{rm:diamLB_DeltaMin_local_delta},
        \[
            \delta_k=\sin\frac{\pi}{2k} \sim \frac{\pi}{2k},\qquad
            \frac{\Delta_k}{\Delta_{\min,k}}
            =\frac{\Delta_k}{\Delta_{\average,k}}=\frac1{\delta_k}.
        \]
        Thus the vertex bound of \Cref{prop:vertexUB_delta_local} exceeds that of \Cref{th:vertexUB} by the factor $\delta_k^{-(n-2)}$, which is exponential in $n$ for every fixed $k\geq2$. The exact number of vertices is $4k\cdot2^{n-2}$, so \Cref{th:vertexUB} has the correct order in $k$. The exact diameter is $2k+n-2$, while \Cref{th:diamUB} and \Cref{th:diamUB_delta_local} give $O(n^3 k\ln(nk))$ and $O(n^2 k\ln(nk))$, respectively.

        Finally, the following family is obtained by subdividing the facets of a cube and projecting the subdivision vertices onto the unit sphere to obtain facet normals. Here the cells are small in all $n-1$ tangential directions, yielding a different dependence on $\delta$. The construction and its proof are given in \Vref{proof:prop:spherical_delta_comparison}.

        \begin{restatable}{proposition}{sphericalDeltaComparison}\label{prop:spherical_delta_comparison}
            For every fixed $n\geq 2$, there exists a sequence of simple polytopes $\PC_q\subseteq\RR^n$, $q\in\ZZ_{\geq 1}$, defined by systems $A_q x\leq b_q$ with unit facet normals, such that
            \begin{gather*}
                \abs{\vertex(\PC_q)}=\Theta_n(q^{n-1}), \qquad
                \diam\bigl(\GC(\PC_q)\bigr)=\Theta_n(q),\\
                \Delta(A_q)=1, \qquad
                \Delta_{\min}(\PC_q,A_q)=\Theta_n(q^{-(n-1)}), \qquad
                \Delta_{\average}(\PC_q,A_q)=\Theta_n(q^{-(n-1)}).
            \end{gather*}
            The largest local distance parameter $\delta_q$ of these systems satisfies $\delta_q=\Theta_n(q^{-1})$. In particular,
            \begin{gather}\label{eq:spherical_delta_ratios}
                \frac{\Delta(A_q)}{\Delta_{\min}(\PC_q,A_q)}
                =\Theta(\delta_q^{-(n-1)}), \qquad
                \frac{\Delta(A_q)}{\Delta_{\average}(\PC_q,A_q)}
                =\Theta(\delta_q^{-(n-1)}),\\
                \abs{\vertex(\PC_q)}=\Theta_n(\delta_q^{-(n-1)}), \qquad
                \diam\bigl(\GC(\PC_q)\bigr)=\Theta_n(\delta_q^{-1}).\notag
            \end{gather}
        \end{restatable}

        For the family in \Cref{prop:spherical_delta_comparison}, both vertex bounds have order $\Theta_n(q^{n-1})$, which is the actual order of the number of vertices. Thus the exponent $n-1$ in \Cref{prop:vertexUB_delta_local} is optimal. For the diameter, \Cref{th:diamUB_delta_local} gives $O_n(q\ln q)$, matching the actual order $\Theta_n(q)$ up to a logarithmic factor. The bound of \Cref{th:diamUB} is $O_n(q^{n-1}\ln q)$, which has a larger power of $q$ when $n\geq3$.

        \Cref{tab:vertex_bounds_comparison,tab:diameter_bounds_comparison} summarize these comparisons together with \Cref{ex:delta_row_scaling}. 
        The asymptotic entries correspond to fixed $n$ and to the corresponding parameter ($k$ or $q$) tending to infinity.
        The results of our \Cref{th:vertexUB} and \Cref{th:diamUB} are highlighted in {\color{red} red}. 

        \begin{table}[!htbp]
            \centering
            \caption{Upper bounds on the number of vertices and the actual vertex counts.}
            \label{tab:vertex_bounds_comparison}
            \footnotesize
            \setlength{\tabcolsep}{2pt}
            \renewcommand{\arraystretch}{1.35}
            \begin{tabular}{@{}l|ccc|cc|c@{}}
                \hline
                Polytope family & $\dfrac{\Delta_k}{\Delta_{\min,k}}$
                & $\dfrac{\Delta_k}{\Delta_{\average,k}}$ & $\delta_k$
                & {\color{red}\hypersetup{linkcolor=red}\Cref{th:vertexUB}} & \Cref{prop:vertexUB_delta_local} & Actual count \\
                \hline
                \Cref{ex:delta_row_scaling}
                & $1$ & $1$ & $1/\sqrt2$
                & {\color{red}$\gamma_n$} & $\gamma_n 2^{(n-1)/2}$ & $5\cdot2^{n-2}$ \\
                \Cref{prop:vertexLB}
                & $\Theta_n(n^k)$ & $\sim\dfrac{n+1}{\gamma_n}n^k$ & $\Theta_n(n^{-k})$
                & {\color{red}$\sim(n+1)n^k$} & $\Theta_n(n^{k(n-1)})$ & $(n+1)n^k$ \\
                \Cref{prop:diamLB_DeltaMin}
                & $\sim2k/\pi$ & $\sim2k/\pi$ & $\sim\pi/(2k)$
                & {\color{red}$\sim2\gamma_n k/\pi$} & $\sim\gamma_n(2k/\pi)^{n-1}$ & $4k\cdot2^{n-2}$ \\
                \Cref{prop:spherical_delta_comparison}
                & $\Theta_n(q^{n-1})$ & $\Theta_n(q^{n-1})$ & $\Theta_n(q^{-1})$
                & {\color{red}$\Theta_n(q^{n-1})$} & $\Theta_n(q^{n-1})$ & $\Theta_n(q^{n-1})$ \\
                \hline
            \end{tabular}
        \end{table}

        \begin{table}[!htbp]
            \centering
            \caption{Upper bounds on the diameter and the actual diameters.}
            \label{tab:diameter_bounds_comparison}
            \footnotesize
            \setlength{\tabcolsep}{2pt}
            \renewcommand{\arraystretch}{1.35}
            \begin{tabular}{@{}l|ccc|cc|c@{}}
                \hline
                Polytope family & $\dfrac{\Delta_k}{\Delta_{\min,k}}$
                & $\dfrac{\Delta_k}{\Delta_{\average,k}}$ & $\delta_k$
                & {\color{red}\hypersetup{linkcolor=red}\Cref{th:diamUB}} & \Cref{th:diamUB_delta_local} & Actual diameter \\
                \hline
                \Cref{ex:delta_row_scaling}
                & $1$ & $1$ & $1/\sqrt2$
                & {\color{red}$O(n^3\ln n)$} & $O(n^2\ln n)$ & $n$ \\
                \Cref{prop:vertexLB}
                & $\Theta_n(n^k)$ & $\sim\dfrac{n+1}{\gamma_n}n^k$ & $\Theta_n(n^{-k})$
                & {\color{red}$O_n(k n^k)$} & $O_n(k n^k)$ & $2^{k+1}-1$ \\
                \Cref{prop:diamLB_DeltaMin}
                & $\sim2k/\pi$ & $\sim2k/\pi$ & $\sim\pi/(2k)$
                & {\color{red}$O(n^3 k\ln(nk))$} & $O(n^2 k\ln(nk))$ & $2k+n-2$ \\
                \Cref{prop:spherical_delta_comparison}
                & $\Theta_n(q^{n-1})$ & $\Theta_n(q^{n-1})$ & $\Theta_n(q^{-1})$
                & {\color{red}$O_n(q^{n-1}\ln q)$} & $O_n(q\ln q)$ & $\Theta_n(q)$ \\
                \hline
            \end{tabular}
        \end{table}
        \FloatBarrier

        The bounds of \Cref{th:vertexUB_strengthened} are at least as good as those of \Cref{th:vertexUB,th:diamUB} in the tables, up to the implicit constants in the diameter estimates. Indeed, for every irredundant representation of $\PC$,
        \[
            \max\left\{1,\frac{\abs{\vertex(\PC)}}{\gamma_n}\right\}
            \leq\kappa_{\PC}
            \leq\frac{\Delta}{\Delta_{\average}}
            \leq\frac{\Delta}{\Delta_{\min}}.
        \]

        In particular, $\kappa_{\PC}=1$ for \Cref{ex:delta_row_scaling}. For the family in \Cref{prop:vertexLB}, the lower bound and the table give $\kappa_{\PC_k}\sim(n+1)n^k/\gamma_n$. Similarly, $\kappa_{\PC_k}=\Theta_n(k)$ for \Cref{prop:diamLB_DeltaMin}, and $\kappa_{\PC_q}=\Theta_n(q^{n-1})$ for \Cref{prop:spherical_delta_comparison}. Thus, for fixed $n$, optimizing the row scaling does not improve the orders of the determinant bounds displayed in the tables.
    \end{ChangedText}

\subsection{Complexity Model Assumptions}\label{sec:assumptions}
All algorithms in this work adhere to the \emph{Word-RAM} computational model. Specifically, we assume that basic arithmetic operations (addition, subtraction, multiplication, and division) on rational numbers of bounded size can be performed in constant $O(1)$ time, where the bound is given by the \emph{word size}.

For polyhedra defined by systems $A x \leq b$ with $A \in \QQ^{m \times n}$ and $b \in \QQ^m$, we set the word size to be a fixed polynomial in 
\[
\lceil \log_2(1+n) \rceil + \lceil \log_2(1+ m) \rceil + \lceil \log_2( 1 + \alpha) \rceil,
\]
where $\alpha$ denotes the maximum absolute value of numerators and denominators in the rational representations of entries of $A$ and $b$.


\subsection{Related Work}\label{sec:related_work}
Assuming that $A$ is integer, the number of vertices of $\PC$ can be bounded in terms of $n$ and $\Delta_1 = \max_{i,j} \abs{A_{i j}}$ according to Gribanov et al. \cite{SparseILP_Gribanov}. Specifically, 
\begin{equation*}
    \abs{\vertex(\PC)} \leq (2n)^n \cdot \Delta_1^n.
\end{equation*}
Combining Hadamard's inequality with \Cref{th:vertexUB} yields
\begin{equation*}
\abs{\vertex(\PC)} \lesssim \frac{1}{\sqrt{\pi n}} \cdot \left(\frac{2\pi}{e}\right)^{n/2} \cdot n^n \cdot \Delta_1^n,
\end{equation*}
which improves the exponential term.

Narayanan et al. \cite{Diam_SpectralApproach} established an improved diameter bound for $\GC(\PC)$ with better dependence on $n$, expressed in terms of subdeterminants of the extended matrix $A_{\text{ext}} = (A\,b)$, which is assumed to be integer-valued. Letting $\Delta_{\text{ext}} = \Delta(A_{\text{ext}})$ and $\Delta_{1,\text{ext}} = \max_{i,j} \abs{(A_{\text{ext}})_{i j}}$, their result states:
\begin{equation*}
    \diam\bigl(\GC(\PC)\bigr) = O\left(n^2 \cdot \Delta_{\text{ext}} \cdot \Delta_1 \cdot \log(m \Delta_{\text{ext}} \Delta_{1,\text{ext}}) \right).
\end{equation*}

A central open question in polyhedral combinatorics, known as the \emph{polynomial Hirsch conjecture}, concerns whether $\diam\bigl(\GC(\PC)\bigr) \leq \poly(m,n)$. This problem stems from the classical \emph{Hirsch conjecture}, which posited that $\diam\bigl(\GC(\PC)\bigr) \leq m - n$. Santos \cite{HirschCounterexample} disproved this conjecture for bounded polyhedra by exhibiting a counterexample with diameter at least $(1+\varepsilon)m$. The current state-of-the-art upper bound, established by Sukegawa \cite{Diameter_ImprovedSubexponent}, is subexponential:
\begin{equation*}
    \diam\bigl(\GC(\PC)\bigr) = (m - n)^{\log_2 O(n/\log n)}.
\end{equation*}
Barnette \cite{Diam_Barnette_1974} and Larman \cite{Diam_Larman_1970} proved a bound linear in $m$:
\begin{equation*}
    \diam\bigl(\GC(\PC)\bigr) = O(2^n m).
\end{equation*}
In the special case of integer polytopes contained in $[0,k]^n$, Deza et al. \cite{Diameter_Zonotopes_part1,Diameter_Zonotopes_part2} established bounds:
\begin{align*}
    \diam\bigl(\GC(\PC)\bigr) &\leq k n, \\
    \text{and } \diam\bigl(\GC(\PC)\bigr) &\geq \left\lfloor \frac{(n+1)k}{2} \right\rfloor \text{ for certain families.}
\end{align*}
Further developments can be found in \cite{Diameter_Zonotopes_part4,Diameter_Zonotopes_part3}.

We should also highlight significant developments regarding the circuit diameter of polyhedra. Recent work has established polynomial upper bounds on the circuit diameter in terms of the bit-encoding length of the system $A x \leq b$. Furthermore, researchers have developed polynomial-time computable circuit-pivot rules that can solve linear programs using only a polynomial number of augmentations. For a comprehensive treatment of these results, we refer to Kafer's monograph \cite{PolyhedralDiameters_Optimization_PHD}. Additional important contributions appear in \cite{CircuitDiam_EdgesVsCircuits,CircuitDiam_OnConjecture,EnumerationOptimizationOverCircuits,CircuitDiam_ViaCircuitImbalances,CircuitDiam_SomeCombPoly,Circuit_AugmentationInLP}.

In the case when $A$ is integer and the parameter $\Delta(A)$ is bounded, the polyhedra defined by systems $A x \leq b$ have many interesting properties in algorithmic perspective. Such polyhedra are also known under the name \emph{$\Delta$-modular polyhedra}. 

By \cite{BimodularStrong}, for $\Delta \leq 2$, integer programming on $\Delta$-modular polyhedra is solvable by a strongly polynomial-time algorithm. By \cite{TwoNonZerosStrong}, for any fixed $\Delta$ and assuming that the matrix $A$ has at most two non-zeros per row, the corresponding ILP is also solvable by a strongly polynomial-time algorithm. Previously, a weaker result was known, due to Alekseev \& Zakharova \cite{AZ}. It states that any ILP with a $\{0,1\}$-matrix $A$, which has at most two non-zeros per row and a fixed value of $\Delta\binom{\BUnit^\top}{A}$, can be solved by a linear-time algorithm. Unfortunately, already for $\Delta = 3$ and without any other restrictions on $A$, the computational complexity status of ILP is unknown. Additionally,  we  note  that,  due  to Bock et al. \cite{StableSetHardness}, there  are  no  polynomial-time  algorithms  for  the ILP problems with $\Delta = \Omega(n^\varepsilon)$, for any $\varepsilon > 0$, unless  the ETH (the Exponential Time Hypothesis) is false.

Significant progress has been made in studying $\Delta$-modular polyhedra with $m = n + k$ facets\footnote{Throughout this work, we equate the number of facets with the number of inequalities in the system $A x \leq b$ defining the polyhedron.} and $\Delta$-modular polyhedra of codimension $k$, particularly when $k$ is bounded. Here, polyhedra of bounded codimension refer to those defined by systems of the form $A x = b$, $x \in \ZZ^n_{\geq 0}$ with $k$ equality constraints. For this class of polyhedra, several computational bounds have been established:
\begin{itemize}
    \item Integer linear programming problems can be solved using 
    \[
        O(\log k)^{2k} \cdot \Delta^2 / 2^{\Omega(\sqrt{\log \Delta})}
    \]
    arithmetic operations \cite{ILP_Delta_CoDim}.

    \item The integer feasibility problem admits an algorithm requiring only
    \[
        O(\log k)^{k} \cdot \Delta \cdot (\log \Delta)^3
    \]
    operations \cite{ILP_Delta_CoDim}.

    \item The lattice point count $\abs{\PC \cap \ZZ^n}$ can be computed in
    \[
        O(n/k)^{2k} \cdot n^3 \cdot \Delta^3
    \]
    operations \cite{SparseILP_Gribanov,HyperAvoiding_NonHomo}. The variant of the counting problem for the parameterized polyhedra is considered in \cite{Parametric_Counting_Grib}.

    \item All vertices of $\conv(\PC \cap \ZZ^n)$ can be enumerated using
    \[
        (k \cdot n \cdot \log \Delta)^{O(k + \log \Delta)}
    \]
    operations \cite{OnCanonicalProblems_Grib}.

    \item For $\Delta$-modular simplices, their width can be computed in $\poly(\Delta,n)$ time, and their unimodular equivalence classes can be enumerated by a polynomial-time algorithm when $\Delta$ is fixed \cite{Width_Grib,WidthConv_Grib,SimplexEquiv_Gribanov}.
\end{itemize}

\section{Preliminaries}\label{sec:prelim}

Let $A \in \RR^{m \times n}$. We denote by:
\begin{itemize}
    \item $A_{i\,j}$ the $(i,j)$-th entry of $A$;
    \item $A_{i\,*}$ its $i$-th row vector;
    \item $A_{*\,j}$ its $j$-th column vector;
    \item $A_{\IC \JC}$ the submatrix consisting of rows and columns of $A$ indexed by $\IC$ and $\JC$ respectively;
    \item Replacing $\IC$ or $\JC$ with $*$ selects all rows or columns, respectively;
    \item When unambiguous, we abbreviate $A_{\IC*}$ as $A_{\IC}$ and $A_{*\JC}$ as $A_{\JC}$.
\end{itemize}

For $p \geq 1$, we define the unit $\ell_p$-ball $\BB_p^n = \{x \in \RR^n \colon \norm{x}_p \leq 1\}$. The parallelepiped spanned by vectors \( v_1, \dots, v_n \in \RR^n \) is denoted \(\paral(v_1,\dots,v_n)\). Given a set \(\KC \subseteq \RR^n\), let
\[
\Delta(\KC) = \sup\left\{\vol \paral(v_1, \dots,v_n) \colon v_1,\dots,v_n \in \KC \right\}.
\]
For a matrix \( A \in \ZZ^{n \times m} \), we denote by \(\Delta(A)\) the maximum absolute value of its full-rank subdeterminants:
\[
\Delta(A) = \max\left\{ \abs{\det A_{\IC \JC}} \colon \IC \in \binom{\intint n}{\rank A},\, \JC \in \binom{\intint m}{\rank A} \right\}.
\]

For a set $\KC \subseteq \RR^n$, a point $v \in \KC$ is called \emph{extreme} if the condition
\begin{equation*}
    \exists x,y \in \KC \colon v \in [x,y] \quad\text{implies}\quad x = v \text{ or } y = v.
\end{equation*}
The set of all extreme points of $\KC$ is denoted by $\extreme(\KC)$. 

\begin{remark}\label{rm:Delta_extreme_rm}
    Since the determinant is multilinear in its columns and by Minkowski's theorem,
    for any convex body $\KC \subseteq \RR^n$, the value $\Delta(\KC)$ is attained at extreme points:
    \begin{equation*}
        \Delta(\KC) = \Delta\bigl(\extreme(\KC)\bigr).
    \end{equation*}
\end{remark}

\begin{definition}\label{def:cone_packing}
    A family $\FS$ of polyhedral cones in $\RR^n$ 
    is said to form a \emph{packing} if, for every pair $\FC_1, \FC_2 \in \FS$,
    \[
    \dim(\FC_1 \cap \FC_2) < n.
    \]
    
\end{definition}

\begin{definition}[Polyhedral Fan]\label{def:fan}
    A family $\FS$ of polyhedral cones in $\RR^n$ is called a \emph{polyhedral fan} if:
    \begin{enumerate}
        \item Every face of any $\FC \in \FS$ also belongs to $\FS$;
        \item For any $\FC_1,\FC_2 \in \FS$, their intersection $\FC_1 \cap \FC_2$ is a common face of both cones;
        \item The family $\FS$ has no inclusion-maximal elements of dimension less than $n$.
    \end{enumerate}
    The $k$-dimensional faces and $f$-vector of $\FS$ are defined respectively by
    \begin{gather*}
        \Gamma_k(\FS) = \left\{ \FC \in \FS \colon \dim(\FC) = k \right\}, \\
        f_k(\FS) = \abs{\Gamma_k(\FS)}.
    \end{gather*}
    Observe that $\Gamma_n(\FS)$ forms a packing. The fan $\FS$ is called simplicial if each element of $\Gamma_n(\FS)$ is a simplex.
\end{definition}

\begin{definition}[Triangulation of a Polyhedral Fan]
    Let $\FS$ be a polyhedral fan in $\RR^n$. We say that a polyhedral fan $\TS$ forms a \emph{triangulation of $\FS$} if $\TS$ is induced by a partition of each $\FC \in \Gamma_n(\FS)$ into simplicial cones. 
\end{definition}

In our work, we establish connections between polyhedral fan properties and matrix properties, where matrix columns represent rays of a polyhedral fan. This motivates our definitions of matrix-based polyhedral fans and cone packings:
\begin{definition}[Polyhedral Fan Based on Matrix Columns]\label{def:matrix_based}
    Let $\FS$ be a polyhedral fan or a packing of polyhedral cones. We say that $\FS$ is \emph{based on the columns of a matrix $A \in \RR^{n \times m}$} if each $\FC \in \FS$ is the conic hull of some column subset of $A$. 
    Abusing notation we also write $\FC$ for the corresponding subset of column indices $\FC \subseteq \intint m$. This way, we have
    \begin{equation*}
        \FC = \cone(A_{* \FC}).
    \end{equation*}
    
    The \emph{underlying non-convex polytope\footnote{We define a non-convex polytope as a finite union of convex polytopes. See \Cref{example:poly_fan} for an example of non-convexity. } of $\FS$} is defined as:
    \begin{equation*}
        \norm{\FS}_A = \bigcup_{\FC \in \FS} \conv(\BZero, A_{* \FC}).
    \end{equation*}
    For a packing $\FS$ of $n$-dimensional simplicial polyhedral cones, we define:
    \begin{gather*}
        \Delta_{\min}(\FS,A) = \min_{\FC \in \FS} \abs{\det A_{* \FC}}, \\
        \Delta_{\average}(\FS,A) = \frac{1}{\abs{\FS}} \sum_{\FC \in \FS} \abs{\det A_{* \FC}}.
    \end{gather*}
    For simplicial fans $\FS$, these extend naturally via $\Gamma_d(\FS)$:
    \begin{gather*}
        \Delta_{\min}(\FS,A) = \Delta_{\min}\bigl(\Gamma_d(\FS),A\bigr), \\
        \Delta_{\average}(\FS,A) = \Delta_{\average}\bigl(\Gamma_d(\FS),A\bigr).
    \end{gather*}
    \begin{ChangedText}
        \noindent For non-simplicial fans, we take the maximum by all possible triangulations:
        \begin{gather*}
            \Delta_{\min}(\FS,A) = \max\left\{\Delta_{\min}(\TS,A) \colon \text{$\TS$ -- is a triangulation of $\FS$}\right\},\\
            \Delta_{\average}(\FS,A) = \max\left\{\Delta_{\average}(\TS,A) \colon \text{$\TS$ -- is a triangulation of $\FS$}\right\}.
        \end{gather*}
    \end{ChangedText}
\end{definition}

\begin{definition}[Graph of a Polyhedral Fan]\label{def:graph_of_fan}
    For a polyhedral fan $\FS$ in $\RR^n$, its \emph{graph} $\GC(\FS)$ is a simple graph with:
    \begin{itemize}
        \item Vertex set $\Gamma_n(\FS)$,
        \item Edge set $\left\{\{\FC_1,\FC_2\} \colon \dim(\FC_1 \cap \FC_2) = n - 1\right\}$.
    \end{itemize}
\end{definition}

\begin{ChangedText}
According to \Cref{def:polyhedral_delta_distance}, we define \emph{$\delta$-distance parameter} for a polyhedral fan $\FS$ as follows.
\begin{definition}\label{def:fan_delta_distance}
    For a polyhedral fan $\FS$ in $\RR^n$, we define $\FS$ to satisfy the local $\delta$-distance property if, for every $\FC \in \Gamma_n(\FS)$ and every lenar independent extreme rays $r_1, r_2, \dots, r_n$ of $\FC$: the rays $r_1, r_2, \dots, r_n$ satisfy the $\delta$-distance property according to \Cref{def:polyhedral_delta_distance}. The largest such $\delta$ is denoted by $\delta(\FS)$.
\end{definition}    
\end{ChangedText}

\begin{ChangedText}
    \begin{example}\label{example:poly_fan}
        As an example of \Cref{def:fan} and \Cref{def:matrix_based}, consider the matrix
        \[
            A = (a_1, a_2, a_3, a_4, a_5) =
            \begin{pmatrix}
                1 & 0 & -1 & -1 &  1\\
                2 & 1 &  2 & -1 & -2
            \end{pmatrix}.
        \]
        Define the polyhedral fan $\FS$ based on the columns of $A$: it is composed of the $1$-dimensional rays induced by the columns of $A$ (the corresponding vectors are depicted in \Cref{fig:skeleton_example}),
        \begin{equation*}
            \Gamma_1(\FS) = \left\{\cone(a_i) \colon i \in \intint{5} \right\},
        \end{equation*}
        and of the $2$-dimensional cones (facets)
        \begin{equation*}
            \Gamma_2(\FS) = \left\{\cone(a_1, a_2),\; \cone(a_3,a_2),\;\cone(a_4, a_5) \right\}.
        \end{equation*}
    
        \noindent The geometric realization $\norm{\FS}_A$ of $\FS$ is depicted in \Cref{fig:poly_fan_example}. It is non-convex, since it is the union of the simplices
        \begin{equation*}
            \conv(\BZero, a_1, a_2),\; \conv(\BZero, a_3, a_2),\; \conv(\BZero, a_4, a_5).
        \end{equation*}
    
        Since $\FS$ is simplicial, we may also evaluate $\Delta_{\min}(\FS,A)$ and $\Delta_{\average}(\FS,A)$ using \Cref{def:matrix_based}:
        \begin{gather*}
            \Delta_{\min}(\FS,A) = \min\{1,1,3\},\\
            \Delta_{\average}(\FS,A) = \frac{1+1+3}{3} = 5/3.
        \end{gather*}

        \noindent The {\bf graph of $\FS$} is isomorphic to the disjoint union of an edge and an isolated vertex. 

        \noindent The {\bf local $\delta$-distance parameter} of $\FS$ is
        \begin{equation*}
            \delta(\FS) = 3/10,
        \end{equation*}
        which is attained by the rays $a_4$ and $a_5$.
    \end{example}
\end{ChangedText}

\begin{figure}
    \centering
    \caption{Base matrix of a polyhedral fan}
    \label{fig:skeleton_example}

    \resizebox{0.6\textwidth}{!}{%
        \begin{tikzpicture}[x=1.5cm,y=1.5cm,>=Stealth]
              \draw[step=.5,gray!12,very thin] (-2.5,-3) grid (2.5,3);
              \draw[step=1,gray!30,thin] (-2.5,-3) grid (2.5,3);
              \draw[->] (-2.5,0)--(2.6,0) node[right] {$x$};
              \draw[->] (0,-3)--(0,3.1) node[right] {$y$};
              \foreach \t in {-2,-1,1,2}{
                \draw (\t,.04)--(\t,-.04) node[below=3pt] {$\t$};
                \draw (.04,\t)--(-.04,\t) node[left=3pt] {$\t$};
              }
              \foreach \x/\y in {0/1,1/2,-1/2,-1/-1,1/-2}{
                \draw[vectorblue,line width=1.2pt,-{Stealth[length=3mm,width=2mm]}]
                  (0,0)--(\x,\y);
              }
              \fill (0,0) circle (1.3pt);
              \node[anchor=east] at (-.18,.14) {$O$};
              \node[vectorblue,above right=4pt] at (0,1) {$(0,1)$};
              \node[vectorblue,above right=4pt] at (1,2) {$(1,2)$};
              \node[vectorblue,above left=4pt] at (-1,2) {$(-1,2)$};
              \node[vectorblue,below left=4pt] at (-1,-1) {$(-1,-1)$};
              \node[vectorblue,above right=4pt] at (1,-2) {$(1,-2)$};
        \end{tikzpicture}
    }%
\end{figure}

\begin{figure}
    \centering
    \caption{Example of a polyhedral fan}
    \label{fig:poly_fan_example}

    \resizebox{0.6\textwidth}{!}{%
        \begin{tikzpicture}[x=1.5cm,y=1.5cm,>=Stealth,
              triangle/.style={fill=trianglegreen,fill opacity=.667,
                draw=triangleedge,line width=.5pt}]
              \draw[step=.5,gray!12,very thin] (-2.5,-3) grid (2.5,3);
              \draw[step=1,gray!30,thin] (-2.5,-3) grid (2.5,3);
              \filldraw[triangle] (-1,2)--(0,1)--(0,0)--cycle;
              \filldraw[triangle] (1,2)--(0,1)--(0,0)--cycle;
              \filldraw[triangle] (-1,-1)--(1,-2)--(0,0)--cycle;
              \draw[->] (-2.5,0)--(2.6,0) node[right] {$x$};
              \draw[->] (0,-3)--(0,3.1) node[right] {$y$};
              \foreach \t in {-2,-1,1,2}{
                \draw (\t,.04)--(\t,-.04) node[below=3pt] {$\t$};
                \draw (.04,\t)--(-.04,\t) node[left=3pt] {$\t$};
              }
              \foreach \x/\y in {0/1,1/2,-1/2,-1/-1,1/-2}{
                \draw[vectorblue,line width=1.2pt,-{Stealth[length=3mm,width=2mm]}]
                  (0,0)--(\x,\y);
              }
              \fill (0,0) circle (1.3pt);
        \end{tikzpicture}
    }
\end{figure}

Recall the definitions of the \emph{normal fan of a polyhedron} and the \emph{face fan of a polyhedron}.
\begin{definition}[Normal Fan of a Polyhedron]\label{def:normal_fan}
    For a polyhedron $\PC \subseteq \RR^n$ defined by $A x \leq b$ ($A \in \RR^{m \times n}$, $b \in \RR^m$), the \emph{normal cone} at vertex $v \in \vertex(\PC)$ is:
    \[
    \FC_v = \cone\left((A_{i *})^\top \colon A_{i *} v = b_i\right).
    \]
    Again we abuse our notation and write $\FC_v$ for the corresponding subset of row indices $\FC_v \subseteq \intint m$. This way, we have
    \[
        \FC_v = \cone\bigl((A_{\FC_v *})^\top\bigr).
    \]
    
    The \emph{normal fan} of $\PC$ denoted by $\normfan(\PC)$ consists of all such $\FC_v$ and their faces. By \Cref{def:matrix_based}, $\normfan(\PC)$ is based on columns of $A^\top$.
    \begin{ChangedText}
        We also recall the definitions of $\Delta_{\average}(\PC, A)$ and $\Delta_{\min}(\PC, A)$ introduced in the Introduction (Equations \eqref{eq:delta_avg_starting} and \eqref{eq:delta_min_starting}). These are direct analogues of the corresponding definitions for polyhedral fans (see \Cref{def:matrix_based}), applied to the normal fan of a polyhedron $\PC$:
        \begin{gather*}
            \Delta_{\average}(\PC,A) = \Delta_{\average}\left(\normfan(\PC),\, A^\top\right),\\
            \Delta_{\min}(\PC,A) = \Delta_{\min}\left(\normfan(\PC),\, A^\top\right).
        \end{gather*}
    \end{ChangedText}
\end{definition}

\begin{ChangedText}
    \begin{definition}[Face Fan of a Polyhedron]
        Let $\PC \subseteq \RR^n$ be a polyhedron with $\BZero \in \relint(\PC)$. According to \cite[Section~7.2]{LecturesOnPolytopes_Ziegler}, we define the \emph{face fan} of $\PC$ as the set of all the cones spanned by proper faces of $\PC$:
        \begin{equation*}
            \facefan(\PC) = \left\{ \cone(\FC) \colon \FC \text{ is a proper face of }\PC \right\}.
        \end{equation*}
    \end{definition}
\end{ChangedText}

\begin{ChangedText}
    In our work we identify the vector space $\RR^n$ with its dual space of linear forms $\bigl(\RR^{n}\bigr)^*$ using the standard canonical isomorphism. Then the \emph{polar} of of a point set can be defined as follows.
    \begin{definition}[Polar Set]
        For a set $\PC \subseteq \RR^n$, the polar set $\PC^{\circ}$ is defined by
        \begin{equation*}
            \PC^\circ = \left\{ x \in \RR^n \colon y^\top x \leq 1,\, \forall y \in \PC \right\}.
        \end{equation*}
    \end{definition}
    
    The fact that the face fan of a polytope containing the origin in its interior coincides with the normal fan of its polar polytope is standard; see, for example, Ziegler \cite{LecturesOnPolytopes_Ziegler}, \emph{Lectures on Polytopes}, Chapter~7 (in particular, Exercise~7.1), and Fulton \cite{Fulton_ToricVarieties}, \emph{Introduction to Toric Varieties}, Section~1.5. Another known fact in convex analysis is that $(\KC^\circ)^\circ = \KC$ for a convex body $\KC$ containing $\BZero$ in its interior. For this and other essential facts about polarity, see, for example, Barvinok \cite{Barv_book_convexity}, \emph{A Course in Convexity}, Chapter~IV. We emphasize these and other standard facts in \Cref{prop:polar_facts}.
    \begin{proposition}\label{prop:polar_facts}
        Let $\PC \subseteq \RR^n$ be a convex polytope such that $\BZero \in \inter(\PC)$. Then the following facts hold:
        \begin{enumerate}
            \item $(\PC^\circ)^\circ = \PC$;
            
            \item $\facefan(\PC) = \normfan(\PC^\circ)$;
            
            \item The graph of $\PC$ is isomorphic to the graph of $\normfan(\PC)$, which in turn is isomorphic to the graph of $\facefan(\PC^\circ)$;
            
            \item $\PC = \conv(v_1, v_2, \dots, v_s)$\footnote{Here $s \geq n+1$, because otherwise $\PC$ would be flat, contradicting $\BZero \in \inter(\PC)$.} if and only if $\PC^\circ$ is represented by the system of linear inequalities $V^\top x \leq \BUnit$, where $V = (v_1, v_2, \dots, v_s)$;

            \item $\PC$ is simple if and only if then $\PC^\circ$ is simplicial.
        \end{enumerate}
    \end{proposition}
\end{ChangedText}

\section{New Upper Bound for the Number of Vertices}\label{sec:vertex_UB}

\subsection{Relation Between the Volume of a Convex Body $\KC$ and $\Delta(\KC)$}

We are motivated by the following proposition, which is a version of the well known relation between $\Delta$-modularity and total $\Delta$-modularity of matrices.
\begin{proposition}\label{prop:totally_unimodular_set}
    Let $\KC \subseteq \RR^n$ be compact, and let $V = (v_1 \dots v_n) \in \RR^{n \times n}$ consist of vectors $v_1, \dots, v_n \in \KC$ satisfying $\Delta(\KC) = \abs{\det V}$. Then the transformed set $\UC = V^{-1} \cdot \KC$ is \emph{totally unimodular} in the following sense: for any $k \geq 1$ and any matrix $U = (u_1 \dots u_k)$ with columns $u_1, \dots, u_k \in \UC$, we have $\Delta(U) \leq 1$. In particular:
    \begin{gather*}
        \UC \subseteq \BB_{\infty}^n, \quad \text{and consequently} \quad \vol \KC \leq 2^n \cdot \Delta(\KC).
    \end{gather*}
\end{proposition}
    

The volume inequality for $\vol \KC$ is non-tight, and a sharper bound follows from the symmetric version of Macbeath's theorem. Recall that the set $\KC$ is called \emph{$o$-symmetric} ($o \in \RR^n$), if for each $x \in \KC$ we have $2o - x \in \KC$.
Let $\PS_m$ denote the class of convex polytopes with at most $m$ vertices.
For a convex body $\KC \subseteq \RR^n$, we define:
\begin{gather*}
    \eta_m(\KC) = \sup\left\{\frac{\vol \PC}{\vol \KC} \colon \PC \subset \KC,\, \PC \in \PS_m\right\}, \\
    \psi_m(\KC) = \sup\left\{\frac{\vol \PC}{\vol \KC} \colon \PC \subset \KC,\, \PC \in \PS_m,\, \text{$\PC$ -- is $o$-symmetric} \right\}.
\end{gather*}

\begin{theorem}[Macbeath \cite{Macbeath}]
    For any convex body $\KC \subseteq \RR^n$ and integer $m \geq 0$,
    \begin{equation*}
        \eta_m(\KC) \geq \eta_m(\BB^n_2).
    \end{equation*}
\end{theorem}
We require the following symmetric version for our analysis, which we will prove below.
\begin{theorem}[A symmetric version of Macbeath’s theorem]\label{th:Macbeath}
    For any $o$-symmetric convex body $\KC \subseteq \RR^n$ and integer $m \geq 0$,
    \begin{equation*}
        \psi_m(\KC) \geq \psi_m(\BB^n_2).
    \end{equation*}
\end{theorem}
To prove \Cref{th:Macbeath}, we follow the original proof of Macbeath in \cite{Macbeath}, and use the following concept. Let $\HC$ be a hyperplane through the origin $o$ and let $\KC$ be an $o$-symmetric convex body, i.e. a compact, convex set with nonempty interior. Any point in $\RR^n$ can be uniquely determined by its signed distance $t$ from $\HC$ and its orthogonal projection $x$ onto $\HC$; in this case we denote the point by $(x,t)$. Let $\proj(\KC)$ denote the orthogonal projection of $\KC$ onto $\HC$, and for any $x \in \proj(\KC)$, define $\KC^+(x) = \max \left\{ t\colon (x,t) \in \KC \right\}$ and $\KC^-(x) = \min \left\{ t\colon (x,t) \in \KC \right\}$. Since $\KC$ is convex, we have that $\KC^- : \proj(\KC) \to \RR$ is a convex function and $\KC^+ : \proj(\KC) \to \RR$ is a concave function.
Define the set
\begin{equation*}
    \SC_{\HC}(\KC) = \left\{ (x,t) \colon x \in \proj(\KC),\, \abs{t} \leq \frac{1}{2} \bigl(\KC^+(x) - \KC^-(x)\bigr) \right\}.
\end{equation*}
This set, which is known to be a convex body \cite[Section 10]{BookSchneider}, is called the \emph{Steiner symmetral} of $\KC$.
Clearly, as $\KC$ is $o$-symmetric, so is $\SC_{\HC}(\KC)$. Our main lemma is the following.
\begin{lemma}\label{lm:Macbeath}
For any $o$-symmetric convex body $\KC$, $m \geq 0$, and hyperplane $\HC$ through $o$, we have
\[
\psi_m(\KC) \geq \psi_m\bigl(\SC_{\HC}(\KC)\bigr).
\]
\end{lemma}

\begin{proof}
Let $\PC \in \PS_m$ be an $o$-symmetric convex polytope inside $\SC_{\HC}(\KC)$ of the largest volume; by compactness, such a polytope exists. Then every vertex $(x_i,t_i)$ of $\PC$, where $i \in \intint m$, lies in the boundary of $\SC_{\HC}(\KC)$, and hence, $t_i = \pm\bigl(\KC^+(x_i) - \KC^-(x_i)\bigr)/2$. Since $\SC_{\HC}(\KC)$ is symmetric to the hyperplane $\HC$, the convex hull of the points $(x_i,-t_i)$, which is the reflection of $\PC$ w.r.t. $\HC$, is also a largest volume $o$-symmetric convex polytope in $\SC_{\HC}(\KC)$. Now, let $\QC$ be the convex hull of the points 
\begin{equation*}
q_i =\left( x_i, t_i + \frac{1}{2} \bigl(\KC^+(x_i)+\KC^-(x_i)\bigr) \right),    
\end{equation*}
and $\RC$ be the convex hull of the points  
\begin{equation*}
    r_i =\left( x_i, - t_i + \frac{1}{2} \bigl(\KC^+(x_i)+\KC^-(x_i)\bigr) \right),
\end{equation*}
respectively, where $i \in \intint m$.
Observe that both $\{q_i\}$ and $\{r_i\}$ belong to $\KC$, and thus, both $\QC,\RC \in \PS_m$ are $o$-symmetric convex polytopes inside $\KC$.

The boundary of $\QC$ can be written as the union of the graphs of two functions $\QC^+ : \proj(\PC) \to \RR$ and $\QC^- : \proj(\PC) \to \RR$, where $\QC^-(x) \leq \QC^+(x)$ for all $x \in \proj(\PC)$, $\QC^-$ is convex and $\QC^+$ is concave. Similarly, the boundary of $\RC$ can be written as the union of the graphs of two functions $\RC^+ : \proj(\PC) \to \RR$ and $\RC^- : \proj(\PC) \to \RR$, where $\RC^-(x) \leq \RC^+(x)$ for all $x \in \proj(\PC)$, $\RC^-$ is convex and $\RC^+$ is concave. On the other hand, by their definitions, we have
\begin{gather*}
    \QC^+(x_i) \geq t_i + \frac{1}{2} \bigl(\KC^+(x_i)+\KC^-(x_i)\bigr) \geq \QC^-(x_i), \quad\text{and}\\
    \RC^+(x_i) \geq -t_i + \frac{1}{2} \bigr(\KC^+(x_i)+\KC^-(x_i)\bigr) \geq \RC^-(x_i).
\end{gather*} 
From this, we obtain that for $i \in \intint m$, 
\begin{equation}\label{eq:QRt_inequality}
    \frac{1}{2} \left( \QC^-(x_i) - \RC^+(x_i) \right) \leq t_i \leq \frac{1}{2} \left( \QC^+(x_i) - \RC^-(x_i) \right).
\end{equation}
Note that, since the function $\QC^- - \RC^+$ is convex and the function $\QC^+ - \RC^-$ is concave, the region $\TC$ defined as
\begin{equation*}
    \TC = \left\{ (x,t) \in \RR^n \colon x \in \proj(\PC),\, \frac{1}{2} \bigl( \QC^-(x) - \RC^+(x) \bigr) \leq t \leq \frac{1}{2} \bigl( \QC^+(x) - \RC^-(x) \bigr) \right\}
\end{equation*}
is convex. By \eqref{eq:QRt_inequality}, the polytope $\PC$ belongs to the region $\TC$. From this, Fubini's theorem yields that
\begin{multline*}
    \vol(\PC) \leq \vol(\TC) = \int_{\proj(\PC)} \frac{1}{2} \left( \QC^+(x) - \RC^-(x) \right) - \frac{1}{2} \left( \QC^-(x) - \RC^+(x) \right) \, dx = \\
    = \frac{1}{2} \int_{\proj(\PC)} \QC^+(x) - \QC^-(x) \, dx + \frac{1}{2} \int_{\proj(\PC)} \RC^+(x) - \RC^-(x) \, dx =\\ 
    = \frac{1}{2} \left( \vol(\QC) +  \vol(\RC) \right),
\end{multline*}
and consequently $\vol(\PC) \leq \vol(\QC)$ or $\vol(\PC) \leq \vol(\RC)$.

\end{proof}

\begin{proof}[Proof of \Cref{th:Macbeath}]
Note that the functional $\psi_m(\cdot)$ is continuous with respect to Hausdorff distance. On the other hand, for any $o$-symmetric convex body $\KC$ and $\varepsilon > 0$ there is a sequence of $o$-symmetric convex bodies $\KC_s$ such that each element can be obtained from $\KC$ by finitely many subsequent Steiner symmetrizations, and $\KC_s$ tends to a Euclidean ball, as $s \to \infty$, with respect to Hausdorff distance \cite[Section 10]{BookSchneider}. Thus, \Cref{lm:Macbeath} implies $\psi_m(\KC) \geq \psi_m(\BB^n_2)$.
\end{proof}
\begin{corollary}\label{cor:vol_Delta_body}
    Let $\KC \subseteq \RR^n$ be a compact $n$-dimensional set containing $\BZero$, then
    \begin{equation*}
        \frac{\vol \KC}{\vol \BB^n_2} \leq \Delta(\KC)
    \end{equation*}
\end{corollary}
\begin{proof}
    We are going to use \Cref{th:Macbeath} setting $m = 2n$ and $o = \BZero$. It is straightforward to see that each $o$-symmetric polytope $\PC$ of positive volume with at most $2n$ vertices has the form $\PC = A \cdot \BB_1^n$ for some non-singular $n \times n$ matrix $A$. Denote $\KC' = \conv\bigl(\KC \cup (-\KC)\bigr)$ and let $A \cdot \BB_1^n$ be an $o$-symmetric polytope inside $\KC'$ of the maximum volume. Note that the corresponding polytope of the maximum volume with respect to $\BB_2^n$ is $\BB_1^n$. Thus, by \Cref{th:Macbeath}, we have
    \begin{equation*}
        \frac{\abs{\det A} \cdot \vol \BB_1^n}{\vol \KC'} \geq \frac{\vol \BB_1^n}{\vol \BB_2^n}.
    \end{equation*}
    Consequently, since $\abs{\det A} = \Delta(\KC')$,
    \begin{equation*}
        \frac{\vol \KC'}{\vol \BB^n_2} \leq \Delta(\KC').
    \end{equation*}
    
    From \Cref{rm:Delta_extreme_rm}, we know that $\Delta(\KC') = \Delta\bigl(\extreme(\KC')\bigr) = \Delta(\KC)$. Thus, we obtain
    \begin{equation*}
        \frac{\vol \KC}{\vol \BB^n_2} \leq \frac{\vol \KC'}{\vol \BB^n_2} \leq \Delta(\KC') = \Delta(\KC),
    \end{equation*}
    which completes the proof.
\end{proof}
The inequality becomes an equality for $\KC = r \cdot \BB^n_2$, $r > 0$. 

\subsection{From an Upper Bound on Polyhedral Fans to an Upper Bound on Vertices}

The following corollary is a discrete variation \Cref{cor:vol_Delta_body}.
\begin{corollary}\label{cor:vol_Delta_fan}
    Let $\FS$ be a packing of $n$-dimensional simplicial polyhedral cones based on the columns of a matrix $A \in \RR^{n \times m}$. Denote $\Delta = \Delta(A)$ and $\Delta_{\average} = \Delta_{\average}(\FS,A)$, then
    \begin{gather}
        \frac{\vol \norm{\FS}_A}{\vol \BB^n_2} \leq \Delta, \quad\text{and} \label{eq:fan_vol_Delta_ineq}\\
        \abs{\FS} \leq n! \cdot \frac{\Delta}{\Delta_{\average}} \cdot \vol \BB^n_2 \sim \frac{1}{\sqrt{n \pi}} \cdot \left(\frac{2 \pi}{e}\right)^{n/2} \cdot n^{n/2} \cdot \frac{\Delta}{\Delta_{\average}}.\label{eq:fan_NVertex_Delta_ineq}
    \end{gather}
\end{corollary}
\begin{proof}
    Denote $\KC = \norm{\FS}_A$. By \Cref{cor:vol_Delta_body} and \Cref{rm:Delta_extreme_rm},
    \begin{equation*}
        \frac{\vol \KC}{\vol \BB_2^n} \leq \Delta(\KC) = \Delta\bigl(\extreme(\KC)\bigr) \leq \Delta.
    \end{equation*}
    The inequality for $\abs{\FS}$ follows directly from the definition of $\Delta_{\average}$: 
    \begin{equation*}
        \vol \KC = \frac{1}{n!} \cdot \abs{\FS} \cdot \Delta_{\average}.
    \end{equation*}
\end{proof}

Now we are ready to give a proof of our main \Cref{th:vertexUB} about the number of vertices.
\vertexUB*

\begin{proof}\label{proof:th:vertexUB}
   Consider the normal fan $\NS(\PC)$ of $\PC$ (see \Cref{def:normal_fan}). Triangulating the elements of $\NS(\PC)$ arbitrarily, we can assume that $\NS(\PC)$ is a simplicial polyhedral fan based on the columns of $A^\top$. Since $\abs{\vertex(\PC)} \leq \abs{\NS(\PC)}$, the theorem follows directly from \Cref{cor:vol_Delta_fan}.
\end{proof}

\section{Note about the Diameter of a Polyhedron}\label{sec:diam_UB}

In this subsection, we prove \Cref{th:diamUB} by combining results from Dadush and H\"ahnle \cite{ShadowSimplexForCurved} with the trick described in \Cref{prop:totally_unimodular_set}. First, we recall several definitions from \cite{SubdeterminantsDiameter} and \cite{ShadowSimplexForCurved}.



According to \cite{ShadowSimplexForCurved}, the diameter of $\tau$-wide polyhedra is bounded by a polynomial function of $n$ and $\tau$.
\begin{theorem}[{\cite[Theorem 3]{ShadowSimplexForCurved}}]\label{th:tau_wide_diam}
    Let $\PC \subseteq \RR^n$ be an $n$-dimensional $\tau$-wide pointed polyhedron. Then the graph of $\PC$ has diameter bounded by $8n/\tau \cdot (1 + \ln 1/\tau)$. Furthermore, a path of this expected length can be constructed via the shadow simplex method.
\end{theorem}

According to \Cref{prop:tau_delta_comparison}, the local $\delta$-distance property gives a lower bound on the normal-cone width $\tau$. In turn, the parameter $\delta$ can be bounded using the following \Cref{lm:transfomed_poly_delta_distance}.
\begin{lemma}\label{lm:transfomed_poly_delta_distance}
    Let $\PC \subseteq \RR^n$ be a pointed $n$-dimensional polyhedron defined by the system $A x \leq b$. Denote $\Delta = \Delta(A)$ and $\Delta_{\min} = \Delta_{\min}(\PC,A)$. Additionally, let $B$ be an $n \times n$ submatrix of $A$ such that $\abs{\det B} = \Delta$. Then, the polyhedron $\PC'$ defined by the system $A B^{-1} x \leq b$ satisfies the local $\delta$-distance property with 
    \begin{equation*}
        \delta \geq \frac{\Delta_{\min}}{n \Delta}.
    \end{equation*}
\end{lemma}

\begin{proof}
    Let $\TS$ be a triangulation of $\NS(\PC)$, for which the value of $\Delta_{\min}(\PC,A)$ is attained. 
    Denote $A' = A B^{-1}$, and observe that for each $\BC \in \TS$,
    \begin{equation}\label{eq:Delta_poly_transform_prop_1}
        \abs{\det A'_{\BC *}} \geq \frac{\Delta_{\min}}{\Delta}.
    \end{equation}
    Additionally, for any $k \in \intint n$ and each square submatrix $C \in \RR^{k \times k}$ of $A'$, we have
    \begin{equation}\label{eq:Delta_poly_transform_prop_2}
        \abs{\det(C)} \leq 1.
    \end{equation}
    

    Fix some $\BC \in \TS$ and denote $M = A'_{\BC *}$. Additionally, fix some $i \in \intint n$, and let $v = (M_{i *})^\top$. The angle $0 \leq \varphi \leq \pi/2$ between $v$ and the subspace spanned by the remaining columns of $M$, i.e. $\linh(M_{\intint n \setminus i})$, can be expressed by the formula
    \begin{equation}\label{eq:cone_angle}
        \sin \varphi = \cos (\pi/2-\varphi) = \frac{\abs{\langle v, u \rangle}}{\norm{v}_2 \norm{u}_2},
    \end{equation}
    where $u$ is a normal vector to $\linh(M_{\intint n \setminus i})$. We can choose $u$ as the $i$-th column of the adjugate of $M$.  From \eqref{eq:Delta_poly_transform_prop_1} and \eqref{eq:Delta_poly_transform_prop_2}, we know that $\norm{u}_{\infty} \leq 1$, $\norm{v}_\infty \leq 1$, and $\abs{\langle v, u \rangle} = \abs{\det M} \geq \frac{\Delta_{\min}}{\Delta}$. Now the proof follows by applying these inequalities to the formula \eqref{eq:cone_angle}.
\end{proof}

As a consequence of \Cref{lm:transfomed_poly_delta_distance}, \Cref{prop:tau_delta_comparison} and \Cref{th:tau_wide_diam}, we get our diameter bound.

\diamUB*


\section{Construction for Lower Bounds}\label{sec:LB}

In this subsection, we construct a sequence of examples described in demonstrating the tightness of inequalities in \Cref{cor:vol_Delta_fan} and \Cref{th:vertexUB}. This sequence simultaneously establishes the lower bounds for polyhedral diameters in \Cref{prop:diamLB_DeltaAvg} and \Cref{prop:diamLB_DeltaMin}. The construction leverages the approximation of the unit sphere by simplicial polytopes.

\subsection{Lower Bounds for Fan Size and Number of Vertices}\label{sec:LB_fan_size}

Our construction involves three interrelated objects: $\left\{ (\FS_k, R_k, \PC_k) \right\}_{k \geq 0}$, where:
\begin{itemize}
    \item $\FS_k$ is a simplicial polyhedral fan generated by columns of $R_k$,
    \item $\PC_k$ is a convex polytope such that the cones induced by $\PC_k$'s faces form $\FS_k$, i.e.,
    \begin{equation*}
        \FS_k = \left\{ \cone(\TC) \colon \TC \text{ is a proper face of } \PC_k \right\}.
    \end{equation*}
\end{itemize}
The construction begins with $\FS_k$, employing iterative subdivision of the facets of a regular simplex.  
\begin{definition}
    For $d \leq n$, let $\SC \subseteq \RR^n$ be a $d$-dimensional simplex with the vertices $v_1, v_2, \dots, v_{d+1}$. Then, the point
    \begin{equation*}
        b = \frac{1}{d+1} \cdot (v_1 + \ldots + v_{d+1})
    \end{equation*}
    is called the \emph{barycenter} of $\SC$.

    We define the \emph{simplicial subdivision operation on $\SC$} as an operation replacing $\SC$ with the set of the following $d+1$ subsimplices:
    \begin{equation*}
        \conv(v_1, \dots, v_{i-1}, b, v_{i+1}, \dots, v_{d+1}), \quad \text{for } i \in \intint{d+1}.
    \end{equation*}
    Denote this operation by $\subdiv(\SC)$. For $k \geq 1$, recursively define the $k$-th subdivision:
    \begin{gather*}
        \subdiv^k(\SC) = \bigcup_{\TC \in \subdiv^{k-1}(\SC)} \subdiv(\TC), \\
        \subdiv^0(\SC) = \{\SC\}.
    \end{gather*}
    For the sake of clarity, we note that the $k$-the subdivision operation $\subdiv^k(\SC)$, started from a $n$-dimensional simplex $\SC$, will replace $\SC$ by a set of $(n+1)^k$ simplices of dimension $n$.
\end{definition}

\begin{ChangedText}

The key property of the subdivision operation is that the vertices of the subdivided simplices form a dense subset of the original simplex $\SC$. Furthermore, as shown in Gottlieb et al.  \cite{Gottlieb2024}, for every $\eta > 0$, the total volume of the subdivision simplices of diameter at least $\eta$ tends to $0$.
\begin{lemma}[Gottlieb et al. \cite{Gottlieb2024}]\label{lm:Gottlieb}
    Let $\SC \subseteq \RR^n$ be a regular $n$-dimensional simplex. Let $k,z$ be integers, and set $s = z n$. Then at most a $\left(z \left(1-e^{-n}\right)^k\right)$-fraction of the simplices in $\subdiv^s(\SC)$ have diameter larger than $\left(\frac{n}{n+1}\right)^z$.
\end{lemma}

\begin{corollary}\label{cor:diam_reciprocal}
    For any $\eta > 0$,
    \begin{equation*}
        \lim_{k \to \infty} \sum_{\substack{\TC \in \subdiv^k(\SC)\\ \diam \TC > \eta }} \vol(\TC) = 0.
    \end{equation*}
\end{corollary}
\begin{proof}
    Let $D = \diam \SC$. Choose an integer $z$ such that
    \begin{equation*}
        D \left( \frac{n}{n+1} \right)^z \leq \eta.
    \end{equation*}
    By \Cref{lm:Gottlieb}, after $z k n$ steps of the subdivision process, the fraction of pieces of diameter at least $\eta$ is at most
    \begin{equation*}
        z \left(1 - e^{-n}\right)^k \xrightarrow[k \to \infty]{} 0.
    \end{equation*}
    All pieces have equal volume, so the same bound holds for the fraction of the occupied volume.
\end{proof}

    

    We now present our explicit iterative construction, which, for each $k \geq 0$, is built from the polyhedral fan $\FS_k$, the matrix $R_k$ (the base matrix of $\FS_k$), and the convex polytope $\PC_k$, where the defining property of $\PC_k$ is that its facet fan (the polyhedral fan generated by the faces of $\PC_k$) coincides with $\FS_k$.
\end{ChangedText}

Let $\PC_0 \subseteq \RR^n$ be a regular $n$-dimensional simplex with the vertices lying in the unit sphere $\sphere^{n-1}$, and let $\TC_1, \dots, \TC_{n+1}$ denote its facets. For each facet $\TC_i$ and iteration $k \geq 0$, we define:
\begin{itemize}
    \item $\BS_{i}^k = \subdiv^k(\TC_i)$ (the $k$-th simplicial subdivision of $\TC_i$),
    \item $\VC_i^k$ = the set of vertices of all simplices in $\BS_{i}^k$,
\end{itemize}
Let $\BS^k = \bigcup_{i=1}^{n+1} \BS_i^k$ and $\VC^k = \bigcup_{i=1}^{n+1} \VC_i^k$. For each $k \geq 0$, we construct:
\begin{itemize}
    \item The polyhedral fan $\FS_k$ with $\Gamma_n(\FS_k) = \left\{ \cone(\SC) \colon \SC \in \BS^k \right\}$,
    \item The matrix $R_k$, whose columns are central projections of $\VC^k$ onto $\sphere^{n-1}$.
\end{itemize} 

Given the construction of $\PC_{k-1}$, we build $\PC_k$ via the following procedure described in \Cref{alg:PCk_construction}.
\begin{algorithm}[H]
    \caption{Construction of $\PC_k$ from $\PC_{k-1}$}\label{alg:PCk_construction}
    \begin{algorithmic}[1]
    \State Initialize $\PC_k \gets \PC_{k-1}$;
    \For{each facet $\TC$ of $\PC_{k-1}$}
        \State Identify the unique $v \in \VC^k$ satisfying $\cone(v) \cap \inter(\TC) \neq \emptyset$;
        \State Select $\alpha > 0$ such that:
        \begin{itemize}
            \item $\alpha v \notin \PC_{k-1}$,
            \item $\TC$ is the sole facet of $\KC$ visible from $\alpha v$;
        \end{itemize}
        \State Update $\PC_k \gets \conv(\PC_k, \alpha v)$;
    \EndFor
    \end{algorithmic}
\end{algorithm}

The construction yields $\PC_k$ with these inductively verifiable properties (algorithm correctness follows from Property 2). To avoid interrupting the flow of the exposition, we defer the formal proof to \Vref{prop:proof:PC properties}.
\begin{restatable}{proposition}{PCproperties}\label{prop:PC_properties}
    For all $k \geq 0$:
    \begin{enumerate}
        \item $\PC_k$ remains convex;
        \item $\FS_k$ is the face fan of $\PC_k$:
            \begin{equation*}
                \FS_k = \facefan(\PC_k);
            \end{equation*}
        \item $\exists$ diagonal matrix $D_k \succ \BZero$ such that the columns of $R_k D_k$ are the vertices of $\PC_k$.
    \end{enumerate}
\end{restatable}

Recall the following definition.
\begin{definition}
    A subset $\YC$ of $\XC \subseteq \RR^n$ is an \emph{$\varepsilon$-dense subset of $\XC$} if 
    \begin{equation*}
        \forall x \in \XC,\, \exists y \in \YC\colon \quad \norm{x-y}_2 \leq \varepsilon.
    \end{equation*}
    In other words, $\YC$ forms an \emph{$\varepsilon$-cover of $\XC$.}
\end{definition}

\begin{lemma}\label{lm:FS_properties}
    The following properties hold:
    \begin{enumerate}
        \item $\lim_{k \to \infty} \vol\left(\norm{\FS_k}_{R_k}\right) = \vol(\BB_2^n)$;
        \item For any $\varepsilon > 0$ there exists $k' = k(\varepsilon)$ such that the set $\norm{\FS_{k'}}_{R_{k'}}$ is an $\varepsilon$-dense subset of $\BB_2^n$;
        \item $\lim_{k \to \infty} \Delta(R_k) = 1$.
    \end{enumerate}
\end{lemma}

\begin{proof}
    \begin{ChangedText}
        The property (2) follows directly from (1), and (3) directly follows from (2). We now focus on the proof of (1).

        Denote $\CCal_k = \norm{\FS_k}_{R_k}$. We need to show that
        \begin{equation*}
            \lim_{k \to \infty} \vol \CCal_k = \vol \BB_2^n. 
        \end{equation*}
        Fix $\eta > 0$. A piece $\QC$ of the subdivision $\BS^k$ of the regular simplex $\PC_0$ is called \emph{large} if $\diam \QC > \eta$, and \emph{small} otherwise. Using \Cref{cor:diam_reciprocal} we will show that the simplices based on small pieces leave only a thin outer shell of $\BB_2^n$ unfilled, while large pieces account for a vanishing volume.

        {\bf 1. Small pieces approximate the sphere closely.} Radial projection 
        \begin{equation*}
            p(x) = \frac{x}{\norm{x}}
        \end{equation*}
        is Lipschitz on $\BB_2^n \setminus \inter\PC_0$, because $\norm{x}_2 \geq 1/n$. Specifically,
        \begin{equation*}
            \norm{p(x)-p(y)}_2 \leq L \norm{x-y}_2, \quad L = 2n.
        \end{equation*}
        Consider a subdivision piece \(Q \in \BS^k\) with 
        \begin{equation*}
            \diam \QC \leq \eta.
        \end{equation*}
        Let \(u\) be a unit direction in its radial cone. The ray through \(u\) meets \(\QC\) at some point \(x \in \QC\), so \(p(x)=u\).
        Every vertex \(v_i\) of \(\QC\) satisfies
        \[
            \norm{p(v_i) - u}_2 \leq L\eta.
        \]
        Therefore every point of the convex hull of the projected vertices is also within \(L\eta\) of \(u\).
        In particular, where the ray through \(u\) meets the projected facet, its distance from \(0\) is at least
        \[
            1-L\eta.
        \]
        Thus, in the cones corresponding to small pieces, the missing volume lies entirely in the shell
        \[
            \BB_2^n \setminus (1-L\eta) \cdot \BB_2^n.
        \]

        {\bf 2. Large pieces occupy a negligible collection of cones.} Let \(\QC \in \BS^k\) be any subdivision piece, and let \(\BC_{\QC}\) be the portion of the unit ball in its radial cone.
        Every point of \(\QC\) has norm at least \(1/n\), so
        \[
            \BC_{\QC} \subseteq n \cdot \conv(\BZero, \QC).
        \]
        Since \(\QC\) lies in a hyperplane at distance \(1/n\) from \(\BZero\),
        \[
            \vol \BC_{\QC} \le n^n \cdot \frac{1}{n^2} \cdot \vol_{n-1} \QC = C \vol_{n-1} \QC, \quad \text{where } C = n^{n-2}.
        \]
        Denote the total area of pieces of $\BS^k$ with diameter greater than \(\eta\) by \(A_k(\eta)\). \Cref{cor:diam_reciprocal} gives
        \[
            \lim_{k \to \infty} A_k(\eta) = 0.
        \]
        Hence their ball portions have total volume at most \(C A_k(\eta)\), which also tends to zero.

        {\bf 3. Combine the two estimates.} For any fixed \(0<\eta<1/L\),
        \[
            0 \leq \vol \BB_2^k - \vol \CCal_k \leq \vol \left( \BB_2^n \setminus (1-L\eta) \cdot \BB_2^n \right) + C A_k(\eta).
        \]
        Letting \(k\to\infty\),
        \[
            \limsup_{k\to\infty} \left( \BB_2^k - \vol \CCal_k \right) \leq \vol \left( \BB_2^n \setminus (1-L\eta) \cdot \BB_2^n \right).
        \]
        Finally, let \(\eta \to 0\). The shell volume tends to zero, yielding
        \[
            \boxed{\lim_{k \to \infty} \vol \CCal_k = \vol \BB_2^n.}
        \]
    \end{ChangedText}

\end{proof}

\begin{remark}
    \begin{ChangedText}
    In the following two figures, \Cref{fig:5_iter} and \Cref{fig:6_iter}, we illustrate $\norm{\FS_k}_{R_k}$ and the first assertion of \Cref{lm:FS_properties} after five and six iterations, respectively. 
\end{ChangedText}

\begin{figure}[h!]
    \centering
    \caption{An illustration of $\norm{\FS_k}_{R_k}$ for $k = 5$ in $\RR^3$}
    \label{fig:5_iter}

    \includegraphics[scale=0.8]{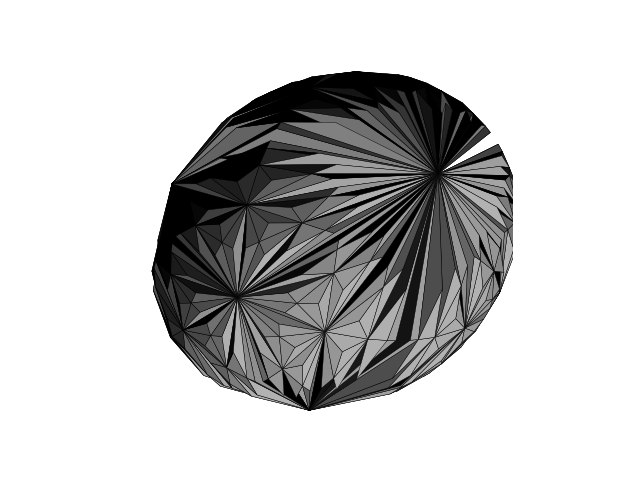}
\end{figure}

\begin{figure}[h!]
    \centering
    \caption{An illustration of $\norm{\FS_k}_{R_k}$ for $k = 6$ in $\RR^3$}
    \label{fig:6_iter}

    \includegraphics[scale=0.5]{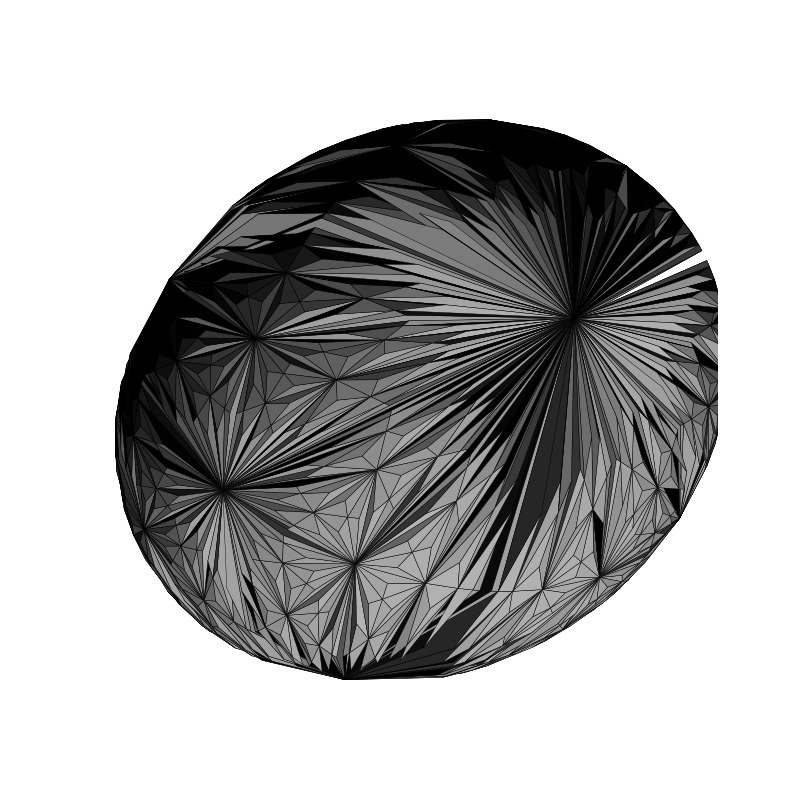}
\end{figure}
\end{remark}

Now we are ready to show that the inequalities \eqref{eq:fan_vol_Delta_ineq}, \eqref{eq:fan_NVertex_Delta_ineq} in \Cref{cor:vol_Delta_fan}, and \Cref{th:vertexUB} are asymptotically tight.
\begin{proposition}\label{prop:tightness_fans}
    For $k \to \infty$, we have
    \begin{gather*}
        \vol\left(\norm{\FS_k}_{R_k}\right) \sim \Delta(R_k) \cdot \vol(\BB_2^n), \quad\text{and}\\
        \abs{\FS_k} \sim n! \cdot \vol(\BB_2^n) \cdot \frac{\Delta(R_k)}{\Delta_{\average}(\FS_k,R_k)}.
    \end{gather*}
\end{proposition}

\begin{proof}
      By \Cref{lm:FS_properties}, passing to the limit $k \to \infty$ yields
        \begin{equation}\label{eq:geometric_realization_and_BB2}
            \vol\left(\norm{\FS_k}_{R_k}\right) \sim \vol(\BB_2^n) \sim \Delta(R_k) \cdot \vol(\BB_2^n).
        \end{equation}
    
    \noindent From the definition of $\Delta_{\average}(\FS_k,R_k)$ and from \eqref{eq:geometric_realization_and_BB2}:
    \begin{equation*}
        \abs{\FS_k} = n! \cdot \vol\left(\norm{\FS_k}_{R_k}\right) \cdot \frac{1}{\Delta_{\average}(\FS_k,R_k)} \sim n! \cdot \vol(\BB_2^n) \cdot \frac{\Delta(R_k)}{\Delta_{\average}(\FS_k,R_k)},
    \end{equation*}
    which proves the second asymptotic equality.
\end{proof}

\vertexLB*

\begin{proof}\label{proof:prop:vertexLB}
    Recall that, by \Cref{prop:PC_properties}, $\PC$ is a polytope containing $\BZero$ in its interior whose vertices are the columns of the matrix $R_k$, scaled by positive factors encoded in the diagonal matrix $D_k$. By polarity (see \Cref{prop:polar_facts}), the polytope $\PC^\circ$ satisfies:
    \begin{equation*}
        \PC_k^{\circ} = \left\{ x \in \RR^n \colon (R_k D_k)^\top x \leq \BUnit \right\} = \left\{ x \in \RR^n \colon R_k^\top x \leq D_k^{-1} \BUnit \right\},
    \end{equation*}
    where vertices of $\PC_k$ correspond bijectively to facets of $\PC_k^{\circ}$. Since $\PC_k$ is simplicial, $\PC_k^{\circ}$ is simple. Using again \Cref{prop:polar_facts}:
    \begin{itemize}
        \item $\FS_k = \normfan(\PC_k^{\circ})$,
        \item $\abs{\vertex(\PC_k^\circ)} = \abs{\FS_k}$,
        \item Since $\PC_k^{\circ}$ is simple, $\Delta_{\average}(\FS_k,R_k) = \Delta_{\average}(\PC_k^{\circ}, R_k^\top )$ (see also \Cref{def:normal_fan}).
    \end{itemize}
    Finally, setting $A_k = R_k^\top$ and $b_k = D_k^{-1} \BUnit$, \Cref{prop:vertexLB} follows immediately from \Cref{prop:tightness_fans}.
\end{proof}

\begin{ChangedText}
\begin{remark}\label{rm:vertexLB_local_delta}
    Let $\delta_k$ be the largest local distance parameter of the fan $\FS_k$ constructed in this subsection. Or, equivalently, the largest local distance parameter of the polytope constructed in the proof of \Cref{prop:vertexLB}. For every fixed $n\geq2$,
    \[
        \delta_k=\Theta_n(n^{-k}),
    \]
    where the implicit constants depend only on $n$.

    Indeed, let $v_1,\dots,v_n$ be the vertices of an arbitrary facet of the original regular simplex $\PC_0$. Setting $c_n=\abs{\det(v_1,\dots,v_n)}>0$, we call this number the \emph{determinant of the facet}. Each subdivision divides the determinant of each corresponding facet-piece by $n$, so after $k$ steps the determinant becomes $c_n/n^k$. Since $\PC_0 \subseteq \PC_k \subset \BB_2^n$, every facet of $\PC_k$ has determinant bounded below by $c_n/n^{k}$. The determinant formula for distance, combined with Hadamard's inequality applied to the remaining unit vectors, yields $\delta_k \geq c_n /n^k$.

    Conversely, retain the vertices $v_1,\dots,v_{n-1}$ and recursively choose the descendant simplex with the remaining vertex
    \[
        w_0=v_n,\qquad w_{j+1}=\frac{w_j+v_1+\cdots+v_{n-1}}{n}.
    \]
    For $H=\linh(v_1,\dots,v_{n-1})$, we have $\dist(w_k,H)=n^{-k}\dist(v_n,H)\leq n^{-k}$. Since $\norm{w_k}_2\geq1/n$, the corresponding set of unit vectors provides $\delta_k\leq n^{1-k}$, proving the claimed order.
\end{remark}
\end{ChangedText}

\subsection{Diameter Lower Bounds}

\subsubsection{Construction in Terms of $\frac{\Delta}{\Delta_{\average}}$}

The construction yielding the diameter lower bound is nearly identical to that yielding the lower bound on the number of vertices or the fan size. Specifically, we reuse the polyhedral fan $\FS_k$ and the base matrix $R_k$. Let $\GC_k$ denote the graph of $\FS_k$ according to \Cref{def:graph_of_fan}. Due to its length, the proof is given in \Vref{lm:proof:DeltaAvg_Diam}.
\begin{restatable}{lemma}{LmDeltaAvgDiam}\label{lm:DeltaAvg_Diam}
    For $n \geq 3$, the graph $\GC_k$ has $(n+1)\cdot n^k$ vertices and diameter $2^{k+1}-1$.
\end{restatable}

\begin{ChangedText}
    Recalling \Cref{prop:tightness_fans} and letting $k \to \infty$, we obtain
    \begin{equation*}
        (n+1)\cdot n^k = \abs{\FS_k} \;\sim\; n! \cdot \vol(\BB_2^n) \cdot \frac{\Delta(R_k)}{\Delta_{\average}(\FS_k,R_k)}.
    \end{equation*}

    \noindent Together with the identity $k+1 = \log_2(1 + \diam \GC_k)$,
    this yields:
    \begin{equation*}
        n^{\log_2 (1+\diam \GC_k)} \sim \eta_n \cdot \frac{\Delta(R_k)}{\Delta_{\average}(\FS_k,R_k)}, \quad \text{where} \quad 
        \eta_n = n! \frac{n}{n+1} \cdot \vol(\BB_2^n).
    \end{equation*}

    \noindent Taking the logarithm and then exponentiating, we obtain
    \begin{gather*}
        \frac{\ln \left(\diam(\GC_k)+1\right)}{\ln 2} - \frac{\ln \left( \eta_n \cdot \frac{\Delta(R_k)}{\Delta_{\average}(\FS_k,R_k)}\right)}{\ln n} \;\to\; 0,\\
        \diam(\GC_k)+1 \;\sim\; \left( \eta_n \cdot \frac{\Delta(R_k)}{\Delta_{\average}(\FS_k,R_k)} \right)^{\frac{\ln 2}{\ln n}}.
    \end{gather*}

    This proves the existence of a diameter lower bound of a polyhedral fan.
    \begin{proposition}\label{prop:fan_diamLB_DeltaAvg}
        There exists a sequence $\left\{(\FS_k, R_k)\right\}_{k \geq 0}$ where:
        \begin{itemize}
            \item $\FS_k$ is a polyhedral fan based on the columns of $R_k$;
            \item The diameter of $\GC(\FS_k)$ satisfies:
            \begin{gather*}
                \diam\bigl(\GC(\FS_k)\bigr) + 1 = 2^{k+1} \;\underset{k \to \infty}{\sim}\; C_n \cdot \left(\frac{\Delta(R_k)}{\Delta_{\average}(\FS_k,R_k)}\right)^{1/\log_2 n},\\
                \text{where}\quad C_n = \left( n! \frac{n}{n+1} \cdot \vol \BB_2^n \right)^{\frac{1}{\log_2 n}} \;\underset{n \to \infty}{\sim}\; \left( \frac{2 \pi n}{e}\right)^{\frac{n}{2 \log_2 n}}.
            \end{gather*}
        \end{itemize}
    \end{proposition}

\end{ChangedText}

Following the proof technique of \Cref{prop:vertexLB}, we obtain a diameter lower bound for polyhedra.

\diamLBDeltaAvg*
\begin{proof}\label{proof:prop:diamLB_DeltaAvg}
    Recalling the proof of \Cref{prop:vertexLB}, consider the simplicial polytope
    \begin{equation*}
        \PC_k^{\circ} = \left\{ x \in \RR^n \colon (R_k D_k)^\top x \leq \BUnit \right\} = \left\{ x \in \RR^n \colon R_k^\top x \leq D_k^{-1} \BUnit \right\}.
    \end{equation*}
    Recall that $\FS_k$ coincides with the normal fan of $\PC_k^{\circ}$, which implies the equality $\diam\bigl(\GC(\PC_k^{\circ})\bigr) = \diam\bigl(\GC(\FS_k)\bigr)$. 
    Set $A_k = R_k^\top$ and $b_k = D_k^{-1} \BUnit$. The result now follows directly from \Cref{prop:fan_diamLB_DeltaAvg}.

\end{proof}

\subsubsection{Linear Construction in Terms of $\frac{\Delta}{\Delta_{\min}}$}

\begin{ChangedText}
In this subsection, we construct examples of fans and polytopes whose graph diameter depends linearly on $\Delta/\Delta_{\min}$ (and on $\Delta/\Delta_{\average}$). The construction comes from a simple idea: to take the direct product of a two-dimensional uniform $k$-gon with a hypercube. The diameter of the $k$-gon depends linearly on $\Delta/\Delta_{\min}$, whereas taking the direct product merely adds $n$ to the diameter. The corresponding graph is obtained by replacing each vertex of the $n$-hypercube with a $k$-cycle and drawing edges between the corresponding vertices of each $k$-cycle. 

In our manuscript, we have chosen to begin with the construction of the polyhedral fan and then pass to the corresponding polytope, following the dual construction. Specifically, we first construct a two-dimensional fan whose rays are generated by a $k$-gon, and then take the direct product with the fan generated by the standard basis vectors; the latter is precisely the normal fan of the hypercube. The formal construction is given in the next \Cref{prop:fan_diam_LB_linear}, while the corresponding polyhedron is constructed in \Cref{prop:diamLB_DeltaMin}, again by duality.



\begin{proposition}\label{prop:fan_diam_LB_linear}
    For every fixed $n \geq 2$, there exists a sequence $\left\{(\FS_k,R_k)\right\}_{k \geq 1}$ of complete simplicial polyhedral fans $\FS_k$ in $\RR^n$, based on the columns of matrices $R_k$, such that
    \begin{gather*}
        \Delta(R_k) = 1, \qquad
        \Delta_{\min}(\FS_k,R_k) = \Delta_{\average}(\FS_k,R_k) = \sin\frac{\pi}{2k},\\
        \diam\bigl(\GC(\FS_k)\bigr) = 2k+n-2
        \;\underset{k \to \infty}{\sim}\;
        \pi \cdot \frac{\Delta(R_k)}{\Delta_{\min}(\FS_k,R_k)},\\
        \lim_{k \to \infty} \frac{\Delta(R_k)}{\Delta_{\min}(\FS_k,R_k)} = \infty.
    \end{gather*}
    
\end{proposition}

\begin{proof}
    Fix $k \geq 1$. Let $e_1,\dots,e_n$ be the standard coordinate vectors of $\RR^n$, and set
    \begin{gather*}
        r_j = \left(\cos\frac{\pi j}{2k},\,\sin\frac{\pi j}{2k},\,0,\dots,0\right)^\top,
        \qquad j=0,\dots,4k-1,\\
        R_k = \bigl(r_0,\dots,r_{4k-1},e_3,-e_3,\dots,e_n,-e_n\bigr).
    \end{gather*}
    Let $\FS_k$ consist of the cones
    \[
        \FC_{j,s} = \cone\bigl(r_j,r_{j+1},s_3e_3,\dots,s_ne_n\bigr),
        \qquad s=(s_3,\dots,s_n)\in\{-1,1\}^{n-2},
    \]
    and all their faces, where the indices of the vectors $r_j$ are taken modulo $4k$. For $n=2$, the coordinate factors and the sign tuple are absent. The planar sectors generated by consecutive rays form a complete simplicial fan in $\RR^2$. Taking their products with the coordinate orthants of $\RR^{n-2}$ gives $\FS_k$, so it is a complete simplicial fan based on the columns of $R_k$.

    Every column of $R_k$ has Euclidean norm $1$, and hence Hadamard's inequality gives $\Delta(R_k)\leq 1$. The columns $r_0=e_1$, $r_k=e_2$, and $e_3,\dots,e_n$ form the identity matrix, so $\Delta(R_k)=1$. Every maximal cone has determinant
    \[
        \abs{\det\bigl(r_j,r_{j+1},s_3e_3,\dots,s_ne_n\bigr)}
        = \sin\frac{\pi}{2k}.
    \]
    Thus, by \Cref{def:matrix_based}, both $\Delta_{\min}(\FS_k,R_k)$ and $\Delta_{\average}(\FS_k,R_k)$ equal $\sin(\pi/(2k))$.

    By \Cref{def:graph_of_fan}, two maximal cones are adjacent if and only if their planar sectors are consecutive and their sign tuples agree, or their planar sectors agree and their sign tuples differ in exactly one coordinate. Consequently, $\GC(\FS_k)$ is the Cartesian product of the cycle on $4k$ vertices and the graph of the $(n-2)$-dimensional cube. A shortest path between $(j,s)$ and $(\ell,t)$ has length
    \[
        \min\left\{\abs{j-\ell},\,4k-\abs{j-\ell}\right\}
        + \abs{\left\{i\in\{3,\dots,n\}\colon s_i\neq t_i\right\}}.
    \]
    Maximizing this distance gives $\diam\bigl(\GC(\FS_k)\bigr)=2k+n-2$. Finally, for fixed $n$,
    \[
        (2k+n-2)\sin\frac{\pi}{2k}\;\longrightarrow\;\pi,
        \qquad
        \frac{1}{\sin(\pi/(2k))}\;\longrightarrow\;\infty,
    \]
    which proves the asserted asymptotic equalities.
\end{proof}

We now prove \Cref{prop:diamLB_DeltaMin} stated in the Introduction. As in \Cref{prop:vertexLB,prop:diamLB_DeltaAvg}, we describe the resulting family by systems of the form $A_k x \leq b_k$, denoting its determinant parameters by $\Delta_k$ and $\Delta_{\min,k}$. As discussed above, duality is used to construct a polytope whose normal fan coincides with the one constructed in \Cref{prop:fan_diam_LB_linear}.

\diamLBDeltaMin*

\begin{proof}\label{proof:prop:diamLB_DeltaMin}
    Let $R_k$ and $\FS_k$ be as in the proof of \Cref{prop:fan_diam_LB_linear}. Set $A_k=R_k^\top$, $b_k=\BUnit$, and
    \[
        \PC_k = \left\{x\in\RR^n\colon R_k^\top x\leq\BUnit\right\}
        = \QC_k\times[-1,1]^{n-2},
    \]
    where $\QC_k\subseteq\RR^2$ is the regular $4k$-gon circumscribed about the unit disk, defined by the first $4k$ inequalities. Thus $\PC_k$ is a full-dimensional simple polytope containing the origin in its interior, and every inequality is facet-defining. At each vertex, the active normals are two consecutive polygon normals $r_j,r_{j+1}$ and one sign choice from each pair $\{e_i,-e_i\}$, $i=3,\dots,n$. Hence $\normfan(\PC_k)=\FS_k$.

    Transposition preserves the absolute values of the full-rank subdeterminants, so $\Delta_k=\Delta(R_k)$. By \Cref{def:normal_fan},
    \begin{gather*}
        \Delta_{\min,k}=\Delta_{\min}(\FS_k,R_k),\\
        \Delta_{\average,k}=\Delta_{\average}(\FS_k,R_k).
    \end{gather*}
    Moreover, $\GC(\PC_k)$ is isomorphic to $\GC(\FS_k)$ by \Cref{prop:polar_facts}. All assertions now follow from \Cref{prop:fan_diam_LB_linear}.
\end{proof}

\begin{remark}\label{rm:diamLB_DeltaMin_local_delta}
    For the fan constructed \Cref{prop:fan_diam_LB_linear}, or, equivalently, for the polytope constructed in \Cref{prop:diamLB_DeltaMin}, the largest local distance parameter is
    \[
        \delta_k=\sin\frac{\pi}{2k}=\Delta_{\min,k}=\Delta_{\average,k}
        \;\underset{k\to\infty}{\sim}\;\frac{\pi}{2k}.
    \]
    Indeed, every feasible basis consists of two consecutive planar unit normals and $n-2$ orthogonal signed coordinate vectors. Each planar normal has distance $\sin(\pi/(2k))$ from the span of the remaining basis vectors, while each coordinate vector has distance $1$.
\end{remark}

Since the determinant ratios tend to infinity, \Cref{prop:diamLB_DeltaMin} rules out any general bound of the form $\diam\bigl(\GC(\PC)\bigr)\leq C_n(\Delta/\Delta_{\min})^\alpha$ with $\alpha<1$, even when $C_n$ may depend arbitrarily on $n$. Thus the exponent of $\Delta/\Delta_{\min}$ in \Cref{th:diamUB} is optimal; the logarithmic factor and the dependence on $n$ are not settled by this example. The same obstruction holds for $\Delta/\Delta_{\average}$ and, by \Cref{prop:fan_diam_LB_linear}, for graphs of polyhedral fans.
\end{ChangedText}

\section{Algorithmic Implications}\label{sec:algo}

\subsection{Computing the Convex Hull}

The following lemma describes an algorithm for constructing a convex hull via the online construction of a normal fan triangulation.  
\begin{lemma}\label{lm:convex_hull_complexity}  
    Let $\PC$ be a polyhedron defined by the system $A x \leq b$, where $A \in \QQ^{m \times n}$, $b \in \QQ^m$, and $\rank A = n$. Let $\tau$ be an upper bound for the sizes of normal fan triangulations of $\PC$. Then, assuming that an initial vertex $v_0$ of $\PC$ is given, there exists an algorithm that computes the vertices of $\PC$ using  
    \begin{equation*}  
        O\bigl(n \cdot m^2 \cdot \tau\bigr) \quad \text{operations.}  
    \end{equation*}  
    If $\PC$ is simple, the complexity bound improves to  
    \begin{equation*}  
        O\bigl(n^2 \cdot m \cdot \tau \bigr) \quad\text{operations}.  
    \end{equation*}  
\end{lemma}  

\begin{proof}  
    Consider the normal fan $\NS(\PC)$ of $\PC$, which consists of the normal cones $\FC_v$ and their faces for $v \in \vertex(\PC)$ (see \Cref{def:normal_fan}). Clearly, constructing a triangulation for a pointed cone in $\RR^{n}$ is equivalent to constructing a triangulation for a point configuration in $\RR^{n-1}$. A triangulation $\TS$ for a point configuration of $p$ points in $\RR^{n}$ with rank $r < n$ can be computed in at most $O(r \cdot (p-r)^2 \cdot |\TS| )$ operations \cite[Lemma~8.2.2]{TriangBook_DeLoera} (see also \cite{Triangulation_Gruzdev}, where an alternative algorithm is proposed). Thus, we can triangulate $\FC_{v_0}$ in $O\bigl(n \cdot (m_{v_0}-n+1)^2 \cdot \abs{\TS_{v_0}}\bigr)$ operations, where $\TS_{v_0}$ denotes the resulting triangulation and $m_{v_0} = \abs{\FC_{v_0}}$.  

    We identify each simplex $\TC \in \TS_{v_0}$ with the corresponding feasible basis of the system $A x \leq b$. For a fixed $\TC$, we can compute the set of neighboring bases in $O(n)$ operations. Some of these bases are feasible, and can correspond to previously unknown vertices of $\PC$. Given a basis $\BC$ of $A x \leq b$, its feasibility can be verified in $O(n^3 + n \cdot m)$ operations. However, if $\BC$ is a neighboring basis of another basis $\BC'$ for which the inverse or LU-decomposition of $A_{\BC'}$ is known, the cost reduces to $O\bigl(n \cdot (n + m)\bigr) = O(n m)$. After processing $v_0$, we proceed to new unexplored vertices adjacent to $v_0$ using BFS or DFS.  

    Thus, the computational cost relative to $v_0$ is bounded by  
    \begin{equation*}  
        O\left( \bigl(n \cdot (m_{v_0}-n+1)^2 + n^2 \cdot m \bigr) \cdot \abs{\TS_{v_0}} \right) \quad\text{operations.}  
    \end{equation*}  
    Since $\abs{\vertex(\PC)} \leq \tau$, the total computational cost is then bounded by  
    \begin{equation*}  
        O\left( n^2 \cdot m \cdot \tau + n \cdot \sum\limits_{v \in \vertex(\PC)} (m_{v}-n+1)^2 \cdot \abs{\TS_{v}} \right) \quad\text{operations.}  
    \end{equation*}  
    Since $m_{v}-n+1 \leq m$, we obtain the general complexity bound. For simple polyhedra, we have $m_v = n$.  
\end{proof}  

Now, it is straightforward to prove \Cref{th:vertex_enumeration}, which states our result about the vertex enumeration. The proof combines the algorithm described above (\Cref{lm:convex_hull_complexity}) with our upper bound on the size of a normal fan (\Cref{cor:vol_Delta_fan}). An initial vertex of $\PC$ can be found using the standard ray-casting procedure from a known rational feasible solution of $A x \leq b$, which, in turn, can be found by a polynomial-time linear programming algorithm.

\vertexEnumeration*

\subsection{Counting Integer Points in Polytopes}

In this subsection, we establish a new complexity bound for counting integer points in $\Delta$-modular polyhedra. Our result builds upon two key ingredients: (1) the upper bound for the size of a polyhedral fan \Cref{cor:vol_Delta_fan}, and (2) the following theorem from \cite{SparseILP_Gribanov}.
\begin{theorem}[Gribanov~et~al. \cite{SparseILP_Gribanov}]\label{th:general_delta_counting}
    Let $\PC \subseteq \RR^n$ be an $n$-dimensional pointed polytope defined by the system $A x \leq b$, where $A \in \ZZ^{m \times n}$ and $b \in \QQ^m$. Denote $\Delta := \Delta(A)$. Assume additionally that a triangulation $\TS$ of the normal fan $\NC(\PC)$ is given. Then, the number $\abs{\PC \cap \ZZ^n}$ can be computed by an algorithm with the computational complexity bound
    \[
        O\left(n^4 \cdot \Delta \cdot \abs{\TS} \cdot \sum\limits_{\TC \in \TS} \abs{\det A_{\TC}}^2\right) \quad \text{operations}.
    \]
\end{theorem}

\begin{remark}
    We note that the complexity bound of \Cref{th:general_delta_counting} was recently improved in the preprint \cite{HyperAvoiding_NonHomo} to
    \begin{equation*}
        O\left(n^3 \cdot \Delta \cdot \abs{\TS} \cdot \sum_{\TC \in \TS} \abs{\det A_{\TC}}^2\right) \quad \text{operations}.
    \end{equation*}
    Moreover, the authors indicate that future work will establish a refined bound of the form
    \begin{equation*}
        n^2 \cdot \Delta \cdot \abs{\TS} \cdot \sum_{\TC \in \TS} \abs{\det A_{\TC}}^2 \cdot \poly(n, \log \Delta) \quad \text{operations}.
    \end{equation*}
\end{remark}

Additionally, we will use the following elementary technical lemma. The proof is given in \Vref{proof:lm:knapsack_ineq}.
\begin{restatable}{lemma}{knapsackInequality}
\label{lm:knapsack_ineq}
    Let $x \in [0,\alpha]^n$ with $\norm{x}_1 \leq \beta$. Then, for any nondecreasing convex function $f \colon \RR_{\geq 0} \to \RR$ satisfying $f(0) = 0$,
    \begin{equation*}
        \sum_{i=1}^n f(x_i) \leq \left\lfloor\frac{\beta}{\alpha}+1\right\rfloor \cdot f(\alpha). 
    \end{equation*}
\end{restatable}

Our main result for integer solutions counting is the following.

\integerCounting*

\begin{proof}
    Let $\TS$ be a given triangulation of the normal fan $\NS(\PC)$. By \Cref{th:general_delta_counting}, the number $\abs{\PC \cap \ZZ^n}$ can be computed in
    \[
        O\left( n^4 \cdot \Delta \cdot \abs{\TS} \cdot \sum_{\TC \in \TS} \abs{\det A_{\TC}}^2 \right)
    \]
    arithmetic operations. \Cref{cor:vol_Delta_fan} yields the inequalities:
    \begin{gather*}
         \sum_{\TC \in \TS} \abs{\det A_{\TC}} \leq n! \cdot \Delta \cdot \vol \BB_2 = O(n)^{n/2} \cdot \Delta, \quad\text{and}\\
         \abs{\TS} \leq n! \cdot \frac{\Delta}{\Delta_{\average}} \cdot \vol \BB_2 = O(n)^{n/2} \cdot \frac{\Delta}{\Delta_{\average}}.
    \end{gather*}
    Applying \Cref{lm:knapsack_ineq}, we obtain
    \[
        \sum_{\TC \in \TS} \abs{\det A_{\TC}}^2 = O(n)^{n/2} \cdot \Delta^2,
    \]
    which establishes the claimed complexity of the counting algorithm. 
    
    \Cref{lm:convex_hull_complexity} guarantees that a triangulation of $\PC$'s normal fan can be constructed in 
    \[
        O(n)^{n/2} \cdot m^2 \cdot \frac{\Delta}{\Delta_{\average}} + T_{\text{LP}}
    \] 
    operations. By Averkov and Schymura \cite{DiffColumnsOther}, we may assume \( m = n^{O(1)} \cdot \Delta \), so the \( m^2 \) term is dominated by the asymptotic complexity. 
    By Dadush et al. \cite{TardosRevisiting_Dadush}, \( T_{\text{LP}} = n^{O(1)} \cdot m \cdot \log (1+\Delta) \). Since \( \frac{\Delta}{\Delta_{\average}} \geq 1 \), the term \( T_{\text{LP}} \) is also asymptotically negligible.
\end{proof}

\section*{Acknowledgements}
\addcontentsline{toc}{section}{Acknowledgements}
The results described in \Cref{sec:vertex_UB}, \Cref{sec:diam_UB} and \Cref{sec:algo} were supported by the grant of the Russian Science Foundation (RScF) No. 24-71-10021. The results of \Cref{sec:LB} were prepared within the framework of the Basic Research Program at HSE University (HSE-BR-2025-080). The work of Zsolt L\'angi are partially supported by the ERC Advanced Grant “ERMiD”, and the NKFIH grant K14754.


We would like to thank the organizers and participants of the Workshop on Open Problems in Combinatorics and Geometry IV for their engagement in discussions on this research.

\section*{Declarations}
\addcontentsline{toc}{section}{Declarations}
\begin{itemize}
    \item Funding: The results described in \Cref{sec:vertex_UB}, \Cref{sec:diam_UB} and \Cref{sec:algo} were supported by the grant of the Russian Science Foundation (RScF) No. 24-71-10021. The results of \Cref{sec:LB} were prepared within the framework of the Basic Research Program at the National Research University Higher School of Economics (HSE). The work of Zsolt L\'angi are partially supported by the ERC Advanced Grant “ERMiD”, and the NKFIH grant K14754.
    \item Ethics approval and consent to participate: Not applicable.
    \item Consent for publication: Not applicable.
    \item Data availability: The manuscript has no associated data.
    \item Materials availability: Not applicable.
    \item Code availability: The manuscript has no associated code.
    \item Author contribution: All authors contributed equally.
    \item Use of AI tools: The main results of this work were obtained without the use of AI theorem-proving tools. However, \Cref{ex:delta_row_scaling} and the examples presented in \Cref{prop:diamLB_DeltaMin,prop:spherical_delta_comparison} were created using AI tools. Additionally, AI was used to improve the quality of the text in the ``\nameref{sec:intro:comparison}'' section.
\end{itemize}

\appendix
\section{Omitted Proofs}

    \subsection{Proof of \Cref{prop:tau_delta_comparison}}\label{proof:prop:tau_delta_comparison}

    Recall the statement of \Cref{prop:tau_delta_comparison}:
    \tauDeltaComparison*

    \begin{ChangedText}
    \begin{proof}
        Set $\delta_i=\dist(r_i,\linh(r_j\colon j\neq i))$, so $\delta=\min_i\delta_i$. Let $a_i$ be the inward unit normal to the facet of $\CCal$ opposite $r_i$. For $s=\sum_i r_i$, we have
        \[
            \left\langle a_i,\frac{s}{\norm{s}_2}\right\rangle
            =\frac{\delta_i}{\norm{s}_2}\geq\frac{\delta}{n}.
        \]
        Hence the ball of radius $\delta/n$ centered at $s/\norm{s}_2$ lies in $\CCal$.

        Conversely, write any unit vector $x\in\CCal$ as $x=\sum_i\lambda_i r_i$, with $\lambda_i\geq0$. Then
        \[
            \alpha\sum_i\lambda_i\leq\langle v,x\rangle\leq1.
        \]
        The distance from $x$ to the $i$-th facet hyperplane is $\langle a_i,x\rangle=\lambda_i\delta_i$. Thus any ball centered at $x$ and contained in $\CCal$ has radius at most $\min_i\lambda_i\delta_i\leq\delta/\alpha$.
    \end{proof}
    \end{ChangedText}

    \subsection{Proof of \Cref{prop:spherical_delta_comparison}}\label{proof:prop:spherical_delta_comparison}

    Recall the statement of \Cref{prop:spherical_delta_comparison}:
    \sphericalDeltaComparison*

    \begin{ChangedText}
    \begin{proof}
        {\bf1. The fan induced by the grid cubes.}
        Subdivide every facet of $[-1,1]^n$ into $(n-1)$-dimensional grid cubes of side length $1/q$, and denote their vertices by
        \[
            \VC_q=\left\{y\in q^{-1}\ZZ^n\colon \norm{y}_{\infty}=1\right\}.
        \]
        To show that the cones over these grid cubes form a polytopal fan, define
        \[
            h_q^0(x)=\norm{x}_{\infty}
            +\sum_{i=1}^n\sum_{k=0}^{q-1}
            \max\left\{\abs{x_i},\frac{k}{q}\norm{x}_{\infty}\right\}.
        \]
        Each summand is convex, positively homogeneous, and piecewise linear. Moreover, $\norm{x}_{\infty}\leq h_q^0(x)\leq(1+nq)\norm{x}_{\infty}$. Thus the inequality $h_q^0(x)\leq1$ defines a bounded polytope $\PC_q^0$ containing the origin in its interior.

        On an original cube facet $\FC$, we have $\norm{y}_{\infty}=1$. The contribution of each free coordinate $y_i$ is therefore $\sum_{k=0}^{q-1}\max\{\abs{y_i},k/q\}$. Its slope changes exactly at the interior grid points $j/q$, $j=-q+1,\dots,q-1$; the term $k=0$ supplies the change at zero. Consequently, the maximal affine regions of $h_q^0$ on $\FC$ are precisely the grid cubes. By homogeneity, these give their cones. Such cones cannot merge across different original facet cones: if a sum of convex functions is affine on a convex set, then every summand is affine there, including $\norm{x}_{\infty}$. For each maximal linearity cone $\CCal$, the section $\{x\in\CCal\colon h_q^0(x)=1\}$ is a facet of $\PC_q^0$ spanning $\CCal$. Thus, $\facefan(\PC_q^0)$ is exactly the fan induced by the grid cubes.

        {\bf2. Triangulation and construction of $\PC_q$.}
        Order all points of $\VC_q$ lexicographically and apply the pulling triangulation\footnote{\color{black}For a polytope $\QC$ with a fixed total order on its vertices, the pulling triangulation is defined recursively. If $\QC$ is a simplex, take it together with all its faces. Otherwise, let $v$ be its first vertex and recursively triangulate every facet not containing $v$, using the restriction of the same order. For each maximal simplex $\TC$ in these facet triangulations, take $\conv(\{v\}\cup\TC)$. These simplices, together with all their faces, form the pulling triangulation of $\QC$. No new vertices are added. Its restriction to any face is the pulling triangulation of that face with the induced vertex order.} to every grid cube with this common order. This is the standard triangulation into $(n-1)!$ simplices: their vertices follow the paths from the coordinatewise smallest corner to the largest corner, increasing one free coordinate by $1/q$ at each step. The triangulations agree on common faces, since their restrictions use the same vertex order. Every simplex has dimension $n-1$, and its cone has dimension $n$.

        A global pulling refinement preserves realizability as the face fan of a polytope; see \cite[Lemmas~2.3.16 and~4.3.12]{TriangBook_DeLoera} and \cite{UnimodularTriangulations_HaaseEtAl}. Hence the refined fan $\FS_q$ is the face fan of a simplicial polytope $\widetilde{\PC}_q$ containing the origin in its interior. Set $\PC_q=\widetilde{\PC}_q^{\circ}$. By \Cref{prop:polar_facts}, $\PC_q$ is simple and $\normfan(\PC_q)=\FS_q$.
        The rays of $\FS_q$ are precisely those generated by $\VC_q$. Let $h_q$ be the support function of $\PC_q$. For $y\in\VC_q$, set $\hat y=y/\norm{y}_2$, take $\hat y^\top$ as the corresponding row of $A_q$, and set $b_q(y)=h_q(\hat y)$. Then $A_qx\leq b_q$ is a description of $\PC_q$ by its facet inequalities with unit normals.

        {\bf3. Estimating the parameters.}
        There are $2n(2q)^{n-1}$ grid cubes, each containing $(n-1)!$ simplices. Therefore,
        \[
            \abs{\vertex(\PC_q)}
            =2n(2q)^{n-1}(n-1)!
            =2^n n!q^{n-1}.
        \]
        Every small simplex from the triangulation of a grid cube has $(n-1)$-dimensional volume $1/\bigl((n-1)!q^{n-1}\bigr)$. Its vertices $y_1,\dots,y_n$ lie in an original facet of $[-1,1]^n$ at distance $1$ from the origin, so the simplex-volume formula gives
        \[
            \abs{\det(y_1,\dots,y_n)}=q^{-(n-1)},
            \qquad
            \abs{\det(\hat y_1,\dots,\hat y_n)}
            =\frac{q^{-(n-1)}}{\prod_{i=1}^n\norm{y_i}_2}
            =\Theta_n(q^{-(n-1)}),
        \]
        since $1\leq\norm{y_i}_2\leq\sqrt n$. These are exactly the feasible determinants, proving the asserted bounds for $\Delta_{\min}$ and $\Delta_{\average}$. Hadamard's inequality gives $\Delta(A_q)\leq1$, and the rows corresponding to $e_1,\dots,e_n\in\VC_q$ give equality.

        The normalized vertices of one small simplex have pairwise distances $O_n(q^{-1})$. For each $i$, we have
        \[
            \dist\bigl(\hat y_i,\linh(\hat y_j\colon j\neq i)\bigr)
            =\frac{\abs{\det(\hat y_1,\dots,\hat y_n)}}
            {\vol_{n-1}\paral(\hat y_1,\dots,\hat y_{i-1},\hat y_{i+1},\dots,\hat y_n)}.
        \]
        The denominator is $O_n(q^{-(n-2)})$: subtract one of its generators from the other $n-2$ generators and apply Hadamard's inequality. Thus every displayed distance is $\Omega_n(q^{-1})$. Conversely, for any $j\neq i$, this distance is at most $\norm{\hat y_i-\hat y_j}_2=O_n(q^{-1})$. It follows that $\delta_q=\Theta_n(q^{-1})$.

        Finally, $\GC(\PC_q)$ is the adjacency graph of the small simplices on the cube boundary. Within each original facet, any two grid cubes can be connected by $O_n(q)$ moves across common faces. Each grid cube contains a bounded number of simplices with a connected adjacency graph, and passing between original facets requires only a bounded number of such passages. Hence the diameter is $O_n(q)$. For the lower bound, take simplices containing two opposite cube vertices. These two cube vertices are at distance $2\sqrt n$, whereas every simplex has diameter at most $\sqrt{n-1}/q$ and adjacent simplices intersect. A path between the chosen simplices therefore requires $\Omega_n(q)$ steps. Consequently, $\diam\bigl(\GC(\PC_q)\bigr)=\Theta_n(q)$.
    \end{proof}
    \end{ChangedText}

    \subsection{Proof of \Cref{prop:PC_properties}}\label{prop:proof:PC_properties}

    Recall the statement of \Cref{prop:PC_properties}:
    \PCproperties*

    \begin{ChangedText}
    \begin{proof}
        {\bf Assertion~(1)} follows from the fact that, according to \Cref{alg:PCk_construction}, $\PC_k$ is obtained from $\PC_{k-1}$ by applying the following iterative procedure. For convenience, we give a simplified description. Setting $\KC := \PC_{k-1}$, for each facet $\TC$ of $\PC_{k-1}$, we choose a point $u$ strictly above it (according to \Cref{alg:PCk_construction}, $u = \alpha v$) such that only the facet $\TC$ of $\KC$ is visible from this point. Then we perform the assignment $\KC := \KC \cup \conv(\TC, u)$, which is equivalent to $\KC := \conv(\KC, u)$. For obvious reasons, after this assignment $u$ becomes a new vertex of $\KC$, and convexity is preserved.

        {\bf Assertion~(2)} follows directly from the construction. Assume that the inductive hypothesis $\FS_{k-1} = \facefan(\PC_{k-1})$ holds. Then each facet $\TC$ of the polytope $\PC_{k-1}$ is pierced by the ray $\conv(v)$ generated by the barycenter $v$ of the vertices of some simplex $\QC$ of the subdivision $\BS_{k-1}$, which is the preimage of $\TC$ under the central projection of the vertices of $\QC$ onto the sphere. After that, the intersection point is lifted above the facet, forming $n$ new facet-simplices, which are the images of the subdivision of the simplex $\QC$ by its barycenter (according to Algorithm~1). The latter, in turn, corresponds exactly to the construction of the fan $\FS_k$ and proves $\FS_{k} = \facefan(\PC_{k})$.

        {\bf Assertion~(3)} follows from assertion~(2), since the vertices of $\PC_k$ lie on the central rays generated by the vertices $\VC_k$ of the subdivision $\BS_k$. The column vectors of the matrix $R_k$ correspond to the intersections of these rays with the sphere.
    \end{proof}
\end{ChangedText}

    \subsection{Proof of \Cref{lm:DeltaAvg_Diam}}\label{lm:proof:DeltaAvg_Diam}

    Recall the statement of \Cref{lm:DeltaAvg_Diam}:
    \LmDeltaAvgDiam*

    \begin{proof}
        The base case $\GC_0 = \KC_{n+1}$ is trivial. For $k \geq 1$, the graph $\GC_k$ can be constructed from $\GC_{k-1}$ by the following procedure.    
        \begin{itemize}
            \item Consider the graph $\HC_k$ with $(n+1)\cdot n^k$ vertices consisting of $(n+1)\cdot n^{k-1}$ disjoint cliques isomorphic to $\KC_n$. Each clique $\CCal_v$ of $\HC_k$ bijectively corresponds to a vertex $v \in \GC_{k-1}$;
            
            \item We define the graph $\GC_k$ to be equal to $\HC_k$ with the following additional edges. For any edge $(u,v)$ of $\GC_{k-1}$, consider the corresponding cliques $\CCal_u$ and $\CCal_v$ in $\GC_k$. Chose a pair of vertices $(\hat u, \hat v) \in \CCal_u \times \CCal_v$ such that $\hat u$ and $\hat v$ have no incident edges in $\GC_k$ except edges of the cliques $\CCal_u$ and $\CCal_v$. Draw an edge $(\hat u, \hat v)$ into $\GC_k$.
        \end{itemize}
        In other words, the set of edges of $\GC_k$, which was added during the second step, forms a perfect matching. Furthermore, each such an edge is drawn if and only if the corresponding edge exists in $\GC_{k-1}$.  
    
        It is straightforward to see that $\GC_k$ is a graph of $\FS_k$. Indeed, the fan $\FS_k$ is constructed from $\FS_{k-1}$ using the subdivision operation applied to each element of $\Gamma_n(\FS_{k-1})$. The subdivision operation replaces each simplicial cone of $\Gamma_n(\FS_{k-1})$ by $n$ simplicial cones which are pair-wise adjacent. Using the graph theoretical language, we just replace each vertex of $\GC_{k-1}$ by a clique. Next, for each pair of adjacent simplicial cones $\FC_1,\FC_2$ in $\Gamma_n(\FS_{k-1})$, there exists exactly one pair of adjacent simplicial cones $\FC_1'$ and $\FC_2'$ such that $\FC_1'$ and $\FC_2'$ are elements of the subdivisions of $\FC_1$ and $\FC_2$, respectively. This exactly corresponds to the described procedure which draws additional edges to $\GC_k$.

        \begin{ChangedText}
            By construction, the number of vertices in $\GC_k$ is $(n+1) \cdot n^k$. Set $d_k = \diam(\GC_k)$. We prove by induction that $d_k = 2^{k+1}-1$ and that $\GC_k$ contains two disjoint edges whose four endpoint cross-distances all equal $d_k$. For $\GC_0 = \KC_{n+1}$, any two disjoint edges have this property, and $d_0=1$.

            Suppose the assertion holds for $\GC_{k-1}$, and put $D=d_{k-1}$. To prove the upper bound, take arbitrary vertices $x,y$ of $\GC_k$. If they belong to the same clique, their distance is at most $1$. Otherwise, let $(v_0,v_1,\dots,v_s)$ be a shortest path in $\GC_{k-1}$ between the vertices corresponding to their cliques, so $x \in \CCal_{v_0}$, $y \in \CCal_{v_s}$, and $s \leq D$. For each $i=1,\dots,s$, denote by $p_iq_i$ the matching edge corresponding to $v_{i-1}v_i$, where $p_i \in \CCal_{v_{i-1}}$ and $q_i \in \CCal_{v_i}$.

            We construct a path from $x$ to $y$ by traversing these matching edges in order. First, join $x$ to $p_1$ inside $\CCal_{v_0}$, using at most one edge. After traversing $p_iq_i$, for $i<s$, join $q_i$ to $p_{i+1}$ inside $\CCal_{v_i}$; this takes one edge because the clique is complete and the matching edges have distinct endpoints. Finally, join $q_s$ to $y$ inside $\CCal_{v_s}$, again using at most one edge. The resulting path uses $s$ matching edges, $s-1$ edges inside intermediate cliques, and at most two edges in the first and last cliques. Its length is therefore at most
            \[
                s+(s-1)+2=2s+1 \leq 2D+1.
            \]
            Since $x,y$ were arbitrary, this proves $d_k \leq 2D+1$.

            For the reverse inequality, choose the two disjoint edges $ab,cd$ supplied by the induction hypothesis in $\GC_{k-1}$: each endpoint of $ab$ has distance $D$ from each endpoint of $cd$. In $\GC_k$, the old edge $ab$ determines a unique edge joining the cliques $\CCal_a$ and $\CCal_b$; denote its endpoints by $\hat a \in \CCal_a$ and $\hat b \in \CCal_b$. Similarly, let $\hat c \in \CCal_c$ and $\hat d \in \CCal_d$ be the endpoints of the unique edge joining $\CCal_c$ and $\CCal_d$. Thus, $\hat a\hat b$ and $\hat c\hat d$ belong to the perfect matching formed by the edges joining distinct cliques.

            Consider a shortest path $P$ from $\hat a$ to $\hat c$, and let $s$ be its number of matching edges. Replacing each vertex of $P$ by the label of its clique and deleting consecutive repetitions gives a walk from $a$ to $c$ in $\GC_{k-1}$ of length $s$. Since the distance from $a$ to $c$ is $D$, we have $s \geq D$. Between two consecutive matching edges, $P$ must use at least one clique edge: otherwise, since each vertex has only one incident matching edge, $P$ would immediately retrace that edge. Hence, if $s \geq D+1$, the length of $P$ is at least $s+(s-1)=2s-1 \geq 2D+1$.

            It remains to consider $s=D$. If the first edge of $P$ were a matching edge, it would necessarily be $\hat a\hat b$. The rest of $P$ would then give a walk from $b$ to $c$ in $\GC_{k-1}$ of length $D-1$, contradicting their distance $D$. Thus, the first edge is a clique edge. Likewise, if the last edge were the matching edge $\hat d\hat c$, the preceding part would give a walk from $a$ to $d$ of length $D-1$, again a contradiction. Therefore, $P$ has a clique edge at each end, in addition to at least $s-1$ clique edges separating its $s$ matching edges. Its length is at least $s+(s-1)+2=2D+1$.

            Interchanging $a,b$ or $c,d$ gives the same lower bound for all four endpoint cross-distances of the two matching edges. Together with the upper bound, this preserves the induction hypothesis and gives $d_k=2d_{k-1}+1$. Since $d_0=1$, we obtain
            \[
                \boxed{\diam(\GC_k) = 2^{k+1} - 1.}
            \]
        \end{ChangedText}
            


                    

                    
    \end{proof}

    \subsection{Proof of \Cref{lm:knapsack_ineq}}\label{proof:lm:knapsack_ineq}

    Recall the statement of \Cref{lm:knapsack_ineq}:
    \knapsackInequality*
    
    \begin{proof}[\Cref{lm:knapsack_ineq}]
        Consider the polytope $\PC = \bigl\{x \in \RR^n \colon x \in [0,\alpha]^n,\, \norm{x}_1 \leq \beta\bigr\}$. Since $F(x) = \sum_{i=1}^n f(x_i)$ is convex, its maximum is attained at a vertex of $\PC$. The vertices of $\PC$ are of two types:
    
    \begin{enumerate}
        \item Vertices coinciding with vertices $v$ of $[0,\alpha]^n$: 
        Since $v \in \PC$, the number of nonzeroes in $v$ is bounded by $\lfloor \beta/\alpha \rfloor$, yielding $F(v) \leq \lfloor \beta/\alpha \rfloor \cdot f(\alpha)$.
        
        \item Intersection points of edges of $[0,\alpha]^n$ with the hyperplane $\sum_{i=1}^n x_i = \beta$:
        Without loss of generality, such vertices have the form
        \begin{equation*}
            v^\top = (\underbrace{\alpha,\dots,\alpha}_{s},\underbrace{0,\dots,0}_{t},x_n),
        \end{equation*}
        where $s+t = n-1$ and $x_n = \beta - s\alpha$. By the monotonicity of $f$,
        \begin{equation*}
            F(v) = s f(\alpha) + f(x_n) \leq \left\lfloor\frac{\beta}{\alpha}+1\right\rfloor \cdot f(\alpha).
        \end{equation*}
    \end{enumerate}
    \end{proof}

\bibliographystyle{sn-mathphys-num}
\bibliography{grib_biblio}

\end{document}